\documentclass{jgcc}


\pdfoutput=1
\usepackage{lastpage}
\jgccdoi{14}{2}{1}{10521}  
\jgccheading{}{\pageref{LastPage}}{}{}%
{Sept.~26,~2022}{Dec.~27,~2022}{}    

\usepackage[utf8]{inputenc}
\usepackage[english]{babel}
\usepackage[T1]{fontenc}
\usepackage{amsmath}
\usepackage{amsthm}
\usepackage{amssymb}
\usepackage{tikz}
\usetikzlibrary{arrows}
\usetikzlibrary{decorations.markings}
\usetikzlibrary{calc}
\usetikzlibrary{patterns}
\usetikzlibrary{calc}
\usepackage{float}
\usepackage{hyperref}

\newcommand{\N}{\mathbb{N}}
\newcommand{\Z}{\mathbb{Z}}
\newcommand{\E}{e}
\newcommand{\BR}{\mathsf{BR}}
\newcommand{\Sol}{\mathsf{sol}}
\newcommand{\DHB}{H_3(\Z)}
\newcommand{\rest}{\mathord\restriction}

\begin{document}

\title{Exponent equations in HNN-extensions}

\author[M.~Figelius]{Michael Figelius}
\address{Department of Electrical Engineering and Computer Science, University of Siegen, Germany}	
\email{figelius@eti.uni-siegen.de}

\author[M.~Lohrey]{Markus Lohrey}
\address{Department of Electrical Engineering and Computer Science, University of Siegen, Germany}	
\email{lohrey@eti.uni-siegen.de}

\begin{abstract}
We consider exponent equations in finitely generated groups. These are equations, where the variables appear as exponents of group elements and take
values from the natural numbers. Solvability of such (systems of) equations has been intensively studied for various classes of groups in recent years.
In many cases, it turns out that the set of all solutions on an exponent equation is a semilinear set that can be constructed effectively. Such groups are
called knapsack semilinear. Examples of knapsack semilinear groups are 
hyperbolic groups, virtually special groups, co-context-free groups and free solvable groups. Moreover, knapsack semilinearity is preserved by many group theoretic
constructions, e.g., finite extensions, graph products, wreath products, amalgamated free products with finite amalgamated subgroups, and  
HNN-extensions with finite associated subgroups. On the other hand, arbitrary HNN-extensions do not preserve knapsack semilinearity. 
In this paper, we consider the knapsack semilinearity of HNN-extensions, where the stable letter $t$ acts trivially by conjugation on 
 the associated subgroup $A$ of the base group $G$.
We show that under some additional technical conditions, knapsack semilinearity transfers
from base group $G$ to the HNN-extension $H$. These additional technical conditions are satisfied in many cases, e.g., when 
$A$ is a centralizer in $G$ or $A$ is a quasiconvex subgroup of the hyperbolic group $G$.
\end{abstract}

\maketitle

\section{Introduction}

For an infinite finitely generated group $G$ we consider equations of the form 
\begin{equation} \label{eq-exp-eq}
h_0 g_0^{x_1} h_1 g_1^{x_2} \cdots h_n g_n^{x_n} h_{n+1} = 1
\end{equation}
where the $g_i$ and $h_i$ are given elements of $G$ and the $x_i$ are variables that range over $\mathbb{N}$ (in this paper, the natural numbers always include $0$). In general, it is allowed that
$x_i = x_j$ for $i \neq j$.
Equations of this form are known as exponent equations and have received a lot
of attention in recent years, see e.g. \cite{BabaiBCIL96,BergstrasserGZ21,bier2021exponential,bogopolski2021notes,DuTr18,FigeliusGLZ20,FiLoZe2021,FrenkelNU14,GanardiKLZ18,KoenigLohreyZetzsche2015a,LOHREY2019,LohreyZ18,LohreyZ20,MiTr17,MyNiUs14}. Several variants and problems have been studied in this context. The most general decision problem is to decide whether a given system of exponential equations has a solution where natural numbers 
are assigned to the variables $x_i$. This problem is known to be decidable in hyperbolic groups \cite{LOHREY2019}, free solvable groups \cite{FigeliusGLZ20},
co-context-free groups (groups where the complement of the word problem is context-free), and virtually special groups (finite extensions
of subgroups of right-angled Artin groups).\footnote{We will state in a moment more general results which allow to decide whether a  
system of exponential equations has a solution in a co-context-free group or a virtually special group.} 
Many groups are known to be virtually special, e.g., Coxeter groups, fully residually free groups, one-relator groups with torsion, and
fundamental groups of hyperbolic 3-manifolds.

A simpler problem is the so-called knapsack problem: In this problem the input is a single equation of the form \eqref{eq-exp-eq},
where all $x_i$ are pairwise different variables, and it is asked whether a solution exists. There are groups with a decidable knapsack problem but an undecidable solvability problem for systems of exponent 
equations. Examples are the discrete Heisenberg group \cite{KoenigLohreyZetzsche2015a} and the Baumslag-Solitar group $\mathsf{BS}(1,2)$ \cite[Theorem E.1]{BeGaZe-arxiv}.
Let us also remark that the variants of these problems, where the variables $x_i$ range over $\mathbb{Z}$ are not harder, since one can replace a power $g_i^{x_i}$ with 
$x_i$ ranging over $\mathbb{Z}$ by $g_i^{x_i} (g^{-1}_i)^{y_i}$ with 
$x_i,y_i$ ranging over $\mathbb{N}$.

Another problem is to describe the set of all solutions of an equation \eqref{eq-exp-eq}. It turned out that for many groups this set is effectively semilinear for every 
exponent equation;\footnote{A subset of $\mathbb{N}^n$ is semilinear if it is a finite union of sets of shifted subsemigroups of $(\mathbb{N}^n,+)$.
Effectively semilinear means that from a given exponent equation one can effectively compute the finitely many vectors from a semilinear description of 
the solution set. More details on these definitions can be found in Section~\ref{sec-semi-lin}.} such groups are called {\em knapsack semilinear}. 
First of all, for a knapsack semilinear group one can decide whether a given system of exponent equations has a solution; see \cite[Lemma 3.2]{GanardiKLZ18}. Hyperbolic groups, virtually special groups, co-context-free groups, and free solvable groups are all knapsack semilinear;
see \cite[Theorem 8.1]{LOHREY2019} for hyperbolic groups,  
\cite{GanardiKLZ18} for free solvable groups, 
\cite{KoenigLohreyZetzsche2015a} for co-context-free groups,\footnote{Knapsack semilinearity of co-context-free groups is not explicitly shown in \cite{KoenigLohreyZetzsche2015a} but follows directly from the decidability proof in \cite{KoenigLohreyZetzsche2015a}  for the knapsack problem in a co-context-free group.} 
and  \cite{FiLoZe2021} for virtually special groups. In fact, it is shown in  \cite{FiLoZe2021}  that the class of knapsack semilinear groups is closed under
finite extensions, graph products, amalgamated free products with finite amalgamated subgroups, and  HNN-extensions with finite associated subgroups.
In addition, it is shown in \cite{GanardiKLZ18} that the class of knapsack semilinear groups is closed under wreath products, which implies that
free solvable groups are knapsack semilinear (by Magnus embedding theorem a free solvable group can be embedded in an iterated wreath product
of free abelian groups and the latter are knapsack semilinear).

In this paper we want to further elaborate HNN-extensions. HNN-extensions are a fundamental operation in all areas of geometric and combinatorial group theory.
A theorem of Seifert and van Kampen links HNN-extensions to
algebraic topology. Moreover, HNN-extensions are used in all modern proofs for the undecidability of the word problem
in finitely presented groups. For a base group $G$ with two isomorphic subgroups 
$A$ and $B$ and an isomorphism $\varphi \colon A \to B$, the corresponding
HNN-extension is the group
\begin{equation} \label{HNN}
H=\langle G,t \mid  t^{-1} a t = \varphi(a) \, (a \in A) \rangle.
\end{equation}
Intuitively, it is obtained by adjoing to $G$ a new generator $t$ (the {\em stable letter}) 
in such a way that conjugation of $A$ by $t$ realizes 
$\varphi$. The subgroups $A$ and $B$ are also 
called the {\em associated subgroups}. Recall from the above discussion that if $G$ is knapsack semilinear and 
$A$ and $B$ are finite then also $H$ is knapsack semilinear  \cite{FiLoZe2021}. For arbitrary HNN-extensions, this
is not true. For instance, the Baumslag-Solitar group $\mathsf{BS}(1,2) = \langle a,t \mid  t^{-1} a t = a^2\rangle$ is not
knapsack semilinear \cite{LohreyZ20} but it is an HNN-extension of the knapsack semilinear group $\langle a \rangle \cong \mathbb{Z}$.
This example shows that we have to drastically restrict HNN-extensions in order to get a transfer result for knapsack semilinearity
beyond the case of finite associated subgroups. In this paper we study HNN-extensions of the form
\begin{equation} \label{HNN-simple}
H=\langle G,t \mid  t^{-1} a t = a \, (a \in A) \rangle,
\end{equation}
where $A \leq H$ is a subgroup. In other words, we take in \eqref{HNN}  for $\varphi : A \to B$ the identity on $A$.
Intuitively: we add to the group $G$ a free generator $t$ together with commutation identities  $at = ta$ for all $a \in A$.
This operation interpolates between the free product $G * \langle t \rangle \cong G * \mathbb{Z}$ and the direct product
$G \times \langle t \rangle \cong G \times \mathbb{Z}$.

Even HNN-extensions of the form \eqref{HNN-simple} with f.g.~$A$ are too general for our purpose: 
if the subgroup membership problem for $A$ is undecidable then $H$ has an undecidable
word problem. Hence, we also need some restriction on the subgroup $A \leq G$. We say that $G$ is knapsack semilinear
relative to the subgroup $A$ if for every expression of the form 
$h_0 g_0^{x_1} h_1 g_1^{x_2} \cdots h_n g_n^{x_n} h_{n+1}$ (with $g_i, h_i \in G$) the set of all tuples
$(c_1, \ldots, c_n) \in \mathbb{N}^n$  such that $h_0 g_0^{c_1} h_1 g_1^{c_2} \cdots h_n g_n^{c_n} h_{n+1} \in A$
is a semilinear set. Our main result states that if the group $G$ is (i) knapsack semilinear as well as (ii) knapsack semilinear relative
to the subgroup $A$, then the HNN-extension $H$ in \eqref{HNN-simple} is knapsack semilinear. In some situations we can even avoid
the explicit assumption that $G$ is knapsack semilinear relative
to the subgroup $A$.  HNN-extensions of the form \eqref{HNN-simple}, where $A$ is the centralizer
of a single element $g \in G$ are known as {\em free rank one extensions of centralizers} and were first studied in \cite{MyasnikovR96} in the 
context of so-called exponential groups. It is easy to observe that if $G$ is knapsack semilinear and $A \leq G$ is the centralizer
of a finite set of elements, then $G$ is also knapsack semilinear relative to $A$. In particular the operation of free rank one extension of centralizers
preserves knapsack semilinearity. A corollary of this result is that every fully residually free group is knapsack semilinear. 
The class of fully residually free groups is exactly the class of all groups that can be constructed from $\mathbb{Z}$ by the following operations:
taking finitely generated subgroups, free products and free rank one extensions of centralizers.
Knapsack semilinearity of fully residually free groups also follows from the fact that every fully residually free group is virtually special \cite{Wis09}.

In the second part of the paper, we study HNN-extensions of the form \eqref{HNN-simple}, where $G$ is a hyperbolic group.
A group is hyperbolic if 
all geodesic triangles in the Cayley-graph are $\delta$-slim for a constant $\delta$. The class of hyperbolic groups has
several alternative characterizations (e.g., it is the class of finitely generated groups with a linear Dehn function), which
gives hyperbolic groups a prominent role in geometric group theory.  Moreover, in a certain probabilistic sense, almost all finitely presented groups are hyperbolic \cite{Gro87,Olsh92}.
 Also from a computational viewpoint, hyperbolic groups have nice properties: it is known that the word problem and the conjugacy 
 problem can be solved in linear time  \cite{EpsteinH06,Hol00}. In \cite{LOHREY2019} it was shown that hyperbolic groups are knapsack semilinear.
 Here we extend this result by showing that a hyperbolic group is knapsack semilinear relative to a quasiconvex subgroup.
Quasiconvex subgroups in hyperbolic groups are known to have nice properties. Many algorithmic problems are decidable for quasiconvex subgroups, including the membership problem \cite{KhMiWe17}, whereas
Rips constructed finitely generated subgroups of hyperbolic groups with an undecidable membership problem \cite{Rip82b}.

A short version of this paper was presented at the conference ISSAC 2022 \cite{FigeliusL22}.

\section{Preliminaries}

In the following three subsections we introduce some definitions concerning semilinear sets, finite automata, and groups.

\subsection{Semilinear sets} \label{sec-semi-lin}

Fix a dimension $d \ge 1$.
We extend the operations of vector addition and multiplication of a vector by a matrix to sets of vectors in the obvious way.
A {\em linear subset} of $\N^d$ is a set of the form
$$
L= L(b,P) = b + P \cdot \N^k,
$$
where $b \in \N^d$ and $P \in \N^{d \times k}$.
A subset $S\subseteq \N^d$ is \emph{semilinear}, if it is a finite union of linear sets. 
The class of semilinear sets is known to be effectively closed under boolean operations; quantitative results on the descriptional
complexity of boolean operations on semi-linear sets can be found in \cite{Beier}.
If a semilinear set $S$ is given as a union $\bigcup_{1 \le i \le k} L(b_i,P_i)$, we call the tuple $\mathcal{R} = (b_1, P_1, \ldots, b_k, P_k)$
a \emph{semilinear representation} of $S$.

In the context of knapsack semilinearity (which will be introduced in a moment), we will consider 
semilinear subsets as sets of mappings $f : \{x_1, \ldots, x_d\} \to \N$ for a finite set of variables
$U = \{x_1, \ldots, x_d\}$. Such a mapping $f$ can be identified with the vector $(f(x_1), \ldots, f(x_d))^{\mathsf{T}}$.
This allows to use all vector operations (e.g. addition and scalar multiplication) on the set
$\N^U$ of all mappings from $U$ to $\N$. In general, if $\ast$ is a binary operation on $\N$ (we only use addition or multiplication for $\ast$)
we denote with $f \ast g$ (for $f,g \in \N^U$) the pointwise extension of the operation $\ast$ to $\N^U$, i.e.,
$(f \ast g)(x) = f(x) \ast g(x)$ for all $x \in U$. Moreover, for mappings
$f \in \N^U$, $g \in \N^V$ with $U \cap V = \emptyset$ we define $f \oplus g \in \N^{U \cup V}$
by $(f \oplus g)(x) = f(x)$ for $x \in U$ and  $(f \oplus g)(y) = g(y)$ for $y \in V$. All operations on
$\N^U$ will be extended to subsets of $\N^U$ in the standard pointwise way.
For $L \subseteq \N^U$ and $V \subseteq U$ we denote with
$L\rest_{V}$ the set $\{ f\rest_{V} \mid f \in L \}$ obtained by restricting every function $f \in L$ to the 
subset $V$ of its domain. It is easy to see that the operations $\oplus$ and $\rest_V$ preserve semilinearity in an effective way.

The semilinear sets are exactly those sets that are definable in first-order logic over the structure $(\mathbb{N},+)$
(the so-called {\em Presburger definable sets}). 
All the above mentioned closure properties of semilinear sets follow from this characterization.
A good survey on semilinear results and Presburger arithmetic with references for the above mentioned results is \cite{Haase18}.

\subsection{Regular languages, rational relations, and Parikh images}

More details on finite automata can be found in the standard textbook \cite{HoUl79}.
Let $\Sigma$ be a finite alphabet of symbols. As usual, $\Sigma^*$ denotes the set of all finite words
over the alphabet $\Sigma$. For a word $w = a_1 a_2 \cdots a_n$ with $a_1, \ldots, a_n \in \Sigma$
we denote with $|w|=n$ the length of $w$. We denote the empty word (the unique word of length $0$) with $\varepsilon$;
in group theoretic contexts we also write $1$ for the empty word. A {\em factor} of a word $w \in \Sigma^*$ is any word
$u$ such that $w = s u v $ for word some words $s,v$.

A {\em finite automaton} over the alphabet $\Sigma$ is a tuple $\mathcal{A} = (Q, I, \delta, F)$, where 
$Q$ is a finite set of states, $I \subseteq Q$ is the set of initial states, $\delta \subseteq Q \times \Sigma \times Q$ 
is the set of transitions, and $F \subseteq Q$ is the set of final states. A word $w = a_1 a_2 \cdots a_n$ is accepted
by $\mathcal{A}$ if there are transitions $(q_{i-1}, a_i, q_i) \in \delta$ for $1 \leq i \leq n$ such that
$q_0 \in I$ and $q_n \in F$. With $L(\mathcal{A})$ (the language accepted by $\mathcal{A}$) we denote the 
set of all words accepted by $\mathcal{A}$. A language $L$ is called {\em regular} if it is accepted by a 
finite automaton. 

We fix an arbitrary enumeration $a_1, \ldots, a_k$ of the alphabet $\Sigma$.
For $w \in \Sigma^*$ and $1 \leq i \leq k$ let $|w|_{a_i}$ be the number of occurrences 
of $a_i$ in $w$. The {\em Parikh image} of $w$ is the tuple
$P(w) = (|w|_{a_1}, \ldots, |w|_{a_k}) \in \mathbb{N}^k$.
For a language $L \subseteq \Sigma^*$ its Parikh image is $P(L) = \{ P(w) \mid w \in L \}$.
The following important result was shown by Parikh \cite{Parikh66}.

\begin{thm} \label{thm-Parikh-regular}
The semilinear sets are exactly the Parikh images of the regular languages. From a given 
finite automaton $\mathcal{A}$ one can compute a semilinear representation of $P(L(\mathcal{A}))$.
\end{thm}
We will also use the following simple lemma:
\begin{lem}[\mbox{c.f.~\cite[Lemma 5.8]{FiLoZe2021}}] \label{2dim-monoid}
Let $p,q,r,s,u,v \in \Sigma^*$.
 Then the set 
$$\{ (x,y) \in \mathbb{N} \times \mathbb{N} \mid pq^xr = su^yv \}$$
is semilinear and a semilinear representation can be computed from $p,q,r,s,u,v$.
\end{lem}
A {\em finite state transducer} $\mathcal{T}$ over the alphabet $\Sigma$ is a tuple $\mathcal{T} = (Q, I, \delta, F)$
where $I$ and $F$ have the same meaning as for a finite automaton and 
$$
\delta \subseteq (Q \times \Sigma \times \{\varepsilon\} \times Q) \cup (Q \times \{\varepsilon\} \times \Sigma \times Q).$$
A pair $(u,v) \in \Sigma^* \times \Sigma^*$ is accepted by  $\mathcal{T}$ if there are transitions $(q_{i-1}, a_i, b_i, q_i) \in \delta$
for $1 \leq i \leq |u|+|v|$ (where $a_i, b_i \in \Sigma \cup \{\varepsilon\}$) such that $u = a_1 \cdots a_{|u|+|v|}$,
$v = b_1 \cdots b_{|u|+|v|}$, $q_0 \in I$ and $q_{|u|+|v|} \in F$. With $R(\mathcal{T})$ we denote the set of all pairs accepted by $\mathcal{T}$.
A relation $R \subseteq  \Sigma^* \times \Sigma^*$ is a {\em rational relation} if it is accepted by a 
finite state transducer.

\subsection{Groups} \label{sec-groups}

For more details on group theory we refer the reader to \cite{LySch77}.
Infinite groups are usually given by presentations. Take an arbitrary non-empty set  $\Omega$ and let $\Omega^{-1} = \{a^{-1} \mid a \in \Omega\}$
be a set of formal inverses such that $\Omega \cap \Omega^{-1} = \emptyset$. Let $\Sigma = \Omega \cup \Omega^{-1}$.
The bijection $a \mapsto a^{-1}$ from $\Omega$ to $\Omega^{-1}$ can be extended to a natural 
involution $w \mapsto w^{-1}$ on $\Sigma^*$. For this we set  
$(a^{-1})^{-1} = a$ for $a \in \Omega$ and $(a_1 \cdots a_n)^{-1} = a_n^{-1} \cdots a_1^{-1}$ for $a_1, \ldots, a_n \in \Sigma$.
A word $w \in \Sigma^*$ is called {\em reduced} if it does not contain
an occurrence of a word $aa^{-1}$ or $a^{-1}a$ ($a \in \Sigma$). Applying the cancellation rules $aa^{-1} \to \varepsilon$ or $a^{-1}a \to \varepsilon$
as long as possible, every word $w \in \Sigma^*$  can be mapped to a unique reduced word $\mathsf{red}(w)$.
The {\em free group} $F(\Omega)$ consists of all reduced words together with the group multiplication $u \cdot v = \mathsf{red}(uv)$
for reduced words $u$ and $v$. The mapping  $\mathsf{red}$ can be also viewed as a monoid morphism from 
$\Sigma^*$ to $F(\Omega)$. For a subset $R \subseteq \Sigma^*$ one defines
the group $\langle \Sigma \mid R \rangle$ as the quotient group $F(\Omega)/N_R$, where $N_R$ is the smallest normal subgroup of $F(\Omega)$
that contains $\mathsf{red}(R) \subseteq F(\Omega)$. In other words, $N_R$ is the intersection of all normal 
subgroups of $F(\Omega)$ that contain $\mathsf{red}(R)$. Clearly, every group is (isomorphic to a group) of the form 
$\langle \Sigma \mid R \rangle$.

Let $G = \langle \Sigma \mid R \rangle$ in the following. If $\Sigma$ is finite then $G$ is called
{\em finitely generated} (f.g. for short) and $\Sigma$ is called a {\em finite symmetric generating set} for $G$.
If both $\Sigma$ and $R$ are finite, then $G$ is called {\em finitely presented}. 
The surjective monoid morphism $\mathsf{red} \colon \Sigma^* \to F(\Omega)$
extends to a surjective monoid morphism $h  \colon \Sigma^* \to G$, called the {\em evaluation morphism}.
For two words $u,v \in \Sigma^*$ we write $u =_G v$ if $h(u)=h(v)$.
For a subset $S \subseteq G$ and $u \in \Sigma^*$
we write $u \in_G S$ if $h(u) \in S$.

\section{Knapsack and exponent equations}

Let $G$ be a finitely generated group with the finite symmetric generating set $\Sigma$.
Moreover, let $X$ be a set of formal variables that take values
from $\N$. For a subset $U\subseteq X$, we call a mapping $\sigma : U \to \N$ a \emph{valuation} for $U$.
An \emph{exponent expression} over $\Sigma$ is a formal expression of the form 
$e = v_0 u_1^{x_1} v_1 u_2^{x_2} v_2  \cdots u_k^{x_k} v_k$
with $k \geq 1$, words $u_i, v_i \in \Sigma^*$ and exponent variables $x_1, \ldots, x_k \in X$. 
Here, we allow $x_i = x_j$ for $i \neq j$. We also write $e(x_1,\ldots,x_k)$ in order to make the exponent variables explicit.
We can assume that $u_i \neq \varepsilon$ for all $1 \le i \le k$.
If every exponent variable occurs at most once in $e$, then $e$ is called a \emph{knapsack expression}. 
Let $X_e= \{ x_1, \ldots, x_k \}$ be the set of exponent variables that occur in $e$.
 For a valuation $\sigma : U \to \N$ such that $X_e \subseteq U$ (in which case we also say
 that $\sigma$ is a valuation for $e$), we define 
 $\sigma(e) = v_0 u_1^{\sigma(x_1)} v_1 u_2^{\sigma(x_2)} v_2  \cdots u_k^{\sigma(x_k)} v_k \in \Sigma^*$.
We say that $\sigma$ is a \emph{$G$-solution} of the equation $e=1$ if $\sigma(e) =_G 1$ holds.
With $\Sol_G(e)$ we denote the set of all $G$-solutions $\sigma : X_e \to \N$ of $e$. We can view $\Sol_G(e)$ as a subset
of $\N^m$ for $m=|X_e| \le k$.
We define {\em solvability of systems of exponent equations over $G$} as the following decision problem:
\begin{description}
\item[Input] A finite list of exponent expressions $e_1,\ldots,e_n$ over $\Sigma$.
\item[Question] Is $\bigcap_{1 \le i \le n} \Sol_G(e_i)$ non-empty?
\end{description}
The {\em knapsack problem for $G$} is the following 
decision problem:
\begin{description}
\item[Input] A single knapsack expression $e$ over $\Sigma$.
\item[Question] Is $\Sol_G(e)$ non-empty?
\end{description}
It is easy to observe that the concrete choice of the generating set $\Sigma$ has no influence
on the decidability and complexity status of these problems.

For the above decision problems one could restrict to exponent expressions of the form 
$e = u_1^{x_1} u_2^{x_2} \cdots u_k^{x_k} v$: for $e = v_0 u_1^{x_1} v_1 u_2^{x_2} v_2   \cdots u_k^{x_k} v_k$
and $$e' = (v_0 u_1 v_0^{-1})^{x_1} (v_0 v_1 u_2 v_1^{-1} v_0^{-1})^{x_2} \cdots (v_0 \cdots v_{k-1} u_k v_{k-1}^{-1} \cdots v_0^{-1})^{x_k}
v_0 \cdots v_k$$
we have $\sigma(e) =_G  \sigma(e')$ for every valuation $\sigma$.

The group $G$ is called {\em knapsack semilinear} if for every knapsack expression $\E$ over $\Sigma$,
the set $\Sol_G(\E)$ is a semilinear set of vectors and a semilinear representation can be effectively computed from $\E$.
This implies that for every exponent expression $\E$ over $\Sigma$,
the set $\Sol_G(\E)$ is semilinear as well and a semilinear representation can be effectively computed from $\E$.
This fact can be easily deduced from the known effective closure properties of semilinear set, see e.g. \cite{FiLoZe2021}.

Since the semilinear sets are effectively closed under intersection, solvability of systems of exponent equations is decidable for every knapsack semilinear group.
As mentioned in the introduction, the class of knapsack semilinear groups is very rich.
Examples of a groups, where knapsack is decidable but solvability of systems of exponent equations
is undecidable are the Heisenberg group $\DHB$ 
(the group of all upper triangular $(3 \times 3)$-matrices over the integers, where all diagonal entries
are $1$) \cite{KoenigLohreyZetzsche2015a} and the Baumslag-Solitar group $\mathsf{BS}(1,2)$ \cite[Theorem~E.1]{BeGaZe-arxiv,LohreyZ20}.
These groups are not knapsack semilinear in a strong sense:
there are knapsack expressions $\E$ such that $\Sol_{\DHB}(\E)$ (resp. $\Sol_{\mathsf{BS}(1,2)}(\E)$) is not semilinear. 

Let $S \subseteq G$. We say that $G$ is \emph{knapsack semilinear relative to $S$} if 
for every knapsack expression $\E$ over $\Sigma$,
the set $\{ \sigma : X_\E \to \N \mid \sigma(\E) \in_G S \}$  is a semilinear set of vectors and a semilinear representation can be effectively computed from $\E$.
We are mainly interested in the case where $S$ is a subgroup of $G$.
For sets $S_1, \ldots, S_k \subseteq G$ we say that $G$ is knapsack semilinear relative to $\{S_1, \ldots, S_k\}$ if for every
$1 \leq i \leq k$, $G$ is knapsack semilinear relative to $S_i$.  Note that $G$ is knapsack semilinear if and only if it is 
knapsack semilinear relative to $1$.

\section{HNN-extensions} 

In this section we introduce the important operation of HNN-extension.
Suppose $G=\langle \Sigma\mid R\rangle$ is a finitely generated group 
with the  finite symmetric generating set $\Sigma = \Omega \cup \Omega^{-1}$ and $R \subseteq \Sigma^*$.
Fix two isomorphic subgroups $A$ and $B$ of $G$ together with an isomorphism $\varphi\colon A\to
B$. Let $t \notin \Sigma$ be a new letter. Then the corresponding \emph{HNN-extension} is the group 
$$H=\langle \Sigma\cup \{t,t^{-1}\} \mid R\cup \{t^{-1}a^{-1}t\varphi(a) \mid a\in A\}\rangle$$ (formally, we identify here every element $g \in A \cup B$
with a word over $\Sigma$ that evaluates to $g$). This group is usually denoted by 
\begin{equation} \label{eq-HNN}
H=\langle G, t \mid t^{-1}at=\varphi(a)~(a\in A)\rangle .
\end{equation}
Intuitively, $H$ is obtained from $G$ by adding a new element $t$ such that
conjugating elements of $A$ with $t$ applies the isomorphism $\varphi$.  Here,
$t$ is called the \emph{stable letter} and the groups $A$ and $B$ are the
\emph{associated subgroups}.   A basic fact about HNN-extensions is that
the group $G$ embeds naturally into $H$~\cite{HiNeNe49}.

In this paper, we consider HNN-extensions 
$H=\langle G, t \mid t^{-1}at=\varphi(a)~(a\in A)\rangle$, 
where $A \leq G$ is a subgroup of $G = \langle \Sigma \mid R \rangle$ 
and $\varphi : A \to A$ is the identity mapping. Thus, $H$ can be written as
\begin{equation} \label{eq-HNN-C}
H=\langle G, t \mid t^{-1}at=a~(a\in A)\rangle .
\end{equation}
Let us fix this HNN-extension for the further consideration.
Let us denote with 
\[ h : (\Sigma \cup\{t,t^{-1}\})^*\to H \] 
the evaluation morphism. 

A word $u\in (\Sigma \cup \{t,t^{-1}\})^*$ is called \emph{Britton-reduced} if
it does not contain a factor of the form $t^{-\alpha} w t^\alpha$ with $\alpha\in\{-1,1\}$,
$w \in \Sigma^*$ and $w \in_G A$. A factor of the form $t^{-\alpha} w t^\alpha$ with $\alpha\in\{-1,1\}$,
$w \in \Sigma^*$ and $w \in_G A$ is also called a {\em pin}, which we can replace by $w$. 
Since this decreases the number of $t$'s in the word, we can reduce every word to
an equivalent Britton-reduced word. We denote the set of all Britton-reduced words in the HNN-extension \eqref{eq-HNN-C} by $\BR(H)$.
For $u\in (\Sigma \cup \{t,t^{-1}\})^*$ we define $\pi_t(u)$ as the projection of the word $u$ onto the alphabet $\{t, t^{-1}\}$
and $\pi_\Sigma(u)$ as the projection of the word $u$ onto the alphabet $\Sigma$.
Britton's lemma states that if $u =_H 1$ ($u\in (\Sigma \cup \{t,t^{-1}\})^*$) then $u$ contains a pin or $u \in \Sigma^*$ and $u =_G 1$.
Note that a consequence of this is that if $u \in_H G$
then $u$ contains a pin or $u \in \Sigma^*$. To see this, note that $u \in_H G$ implies that 
$u v =_H 1$ for a word $v \in \Sigma^*$. Britton's lemma implies that $uv$ must contain a pin (i.e., $u$ must contain a pin) or
$u v \in \Sigma^*$ (i.e., $u \in \Sigma^*$). In particular, a Britton-reduced
word that contains $t^{\pm 1}$ cannot represent an element of the base group $G$.

Note that $H/N(t) \cong G$, where $N(t)$ is the smallest normal subgroup of $H$ containing $t$.
By $\pi_G: H \to G$ we denote the canonical projection. We have
$\pi_G(g_0 t^{\delta_1}g_1 \cdots t^{\delta_k}g_k) = g_0 g_1 \cdots g_k$ for $g_0, \ldots, g_k \in G$.
Hence, on the level of words, $\pi_G$ is computed by the projection $\pi_\Sigma : (\Sigma \cup \{t,t^{-1}\})^* \to \Sigma^*$.

\begin{lem} \label{lemma-id-in-H}
Let $w \in (\Sigma \cup \{t,t^{-1}\})^*$. Then $w =_H 1$ if and only if $w \in_H G$ and $\pi_\Sigma(w) =_G 1$.
\end{lem}

\begin{proof}
If $w =_H 1$, i.e., $h(w) = 1$, then clearly $w \in_H G$. Moreover, by Britton's lemma, $w$ can be reduced to a word from $\Sigma^*$ using Britton reduction.
But this word must be $\pi_\Sigma(w)$. Hence, we have $\pi_\Sigma(w) =_G 1$. On the other hand, if $w \in_H G$ and $\pi_\Sigma(w) =_G 1$,
then, again, $w$ can be reduced to $\pi_\Sigma(w) =_G 1$ using Britton reduction, which implies $w =_H 1$.
\end{proof}
We will also need the following lemma (cf. Lemma 2.3 of \cite{HauLo11}):
\begin{lem}\label{red uv}
Assume that $u=u_0 t^{\delta_1}u_1 \cdots t^{\delta_k}u_k$ and
$v=v_0 t^{\varepsilon_1}v_1 \cdots t^{\varepsilon_\ell} v_\ell$ are Britton-reduced words with $u_i, v_j \in \Sigma^*$. Let $0 \le m \leq \max\{k,\ell\}$ be the largest number such that
\begin{itemize}
\item $\delta_{k-i} = -\varepsilon_{i+1}$ for all $0 \leq i \leq m-1$ and
\item $u_{k-i+1} \cdots u_k v_0 \cdots v_{i-1} \in_G A$
for all $0 \leq i \leq m$ (for $i=0$ this condition is trivially satisfied).
\end{itemize}
Then
$w := u_0 t^{\delta_1}u_1 \cdots t^{\delta_{k-m}} (u_{k-m} \cdots u_k v_0 \cdots v_{m}) t^{\varepsilon_{m+1}}v_{m+1} \cdots t^{\varepsilon_\ell}v_\ell$
is a Britton-reduced word with $w =_H uv$.
\end{lem}
The above lemma is visualized in Figure~\ref{fig-peak-elimination}, where $$a = u_{k-m+1} \cdots u_k v_0 \cdots v_{m-1} \in_H A.$$

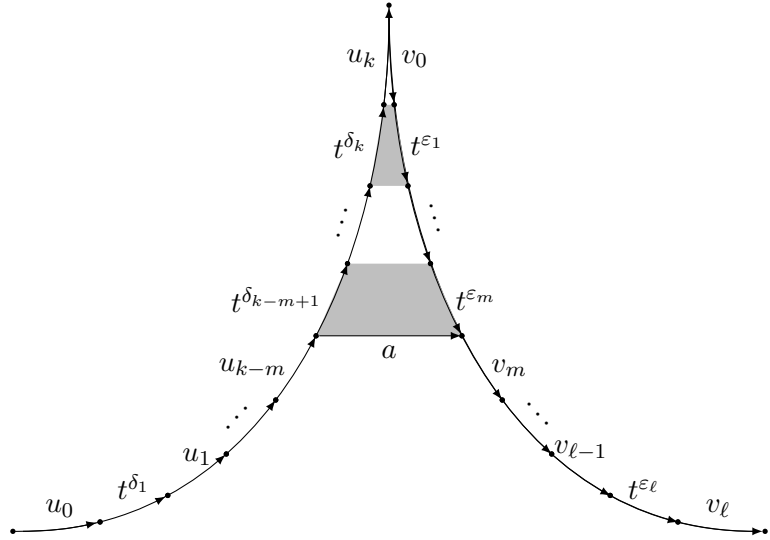
\begin{figure}[t]
\begin{center}
\begin{tikzpicture}[x=0.5cm,y=0.5cm]
 \node (A) at (-10,0) {};
 \node (J) at (0,14) {};
 \node (S) at (10,0) {};
 
  \tikzset{>=latex}
\path[->,shorten >= -3pt, shorten <= -3pt, nodes={circle,fill=black,inner sep=0pt,minimum size=2pt}] (A) to[out=0,in=270]
  node[pos=0.11] (B) {}
  node[pos=0.22] (C) {}
  node[pos=0.33] (D) {}
  node[pos=0.44] (E) {}
  node[pos=0.55] (F) {}
  node[pos=0.66] (G) {}
  node[pos=0.77] (H) {}
  node[pos=0.88] (I) {}
  (J);
  
\draw[->,shorten >= -3pt, shorten <= -3pt, nodes={circle,fill=black,inner sep=0pt,minimum size=2pt}] (J) to[out=270,in=180]
  node[pos=0.12] (K) {}
  node[pos=0.23] (L) {}
  node[pos=0.34] (M) {}
  node[pos=0.45] (N) {}
  node[pos=0.56] (O) {}
  node[pos=0.67] (P) {}
  node[pos=0.78] (Q) {}
  node[pos=0.89] (R) {}
  (S);
  
 \path[fill=gray!50,opacity=1.0]
 (F.center) to (G.center) to (M.center) to (N.center)
 (H.center) to (I.center) to (K.center) to (L.center);
 
\draw[->,shorten <= -3pt] (A) to[bend right=6] (B);

\begin{scope}[decoration={
    markings,
    mark=at position 0.211 with {\arrow{>}}}
    ] 
    \path[postaction={decorate},shorten >= -3pt, shorten <= -3pt] (A) to[out=0,in=270] (J);
\end{scope}

\begin{scope}[decoration={
    markings,
    mark=at position 0.315 with {\arrow{>}}}
    ] 
    \path[postaction={decorate},shorten >= -3pt, shorten <= -3pt] (A) to[out=0,in=270] (J);
\end{scope}

\begin{scope}[decoration={
    markings,
    mark=at position 0.42 with {\arrow{>}}}
    ] 
    \path[postaction={decorate},shorten >= -3pt, shorten <= -3pt] (A) to[out=0,in=270] (J);
\end{scope}

\begin{scope}[decoration={
    markings,
    mark=at position 0.527 with {\arrow{>}}}
    ] 
    \path[postaction={decorate},shorten >= -3pt, shorten <= -3pt] (A) to[out=0,in=270] (J);
\end{scope}

\begin{scope}[decoration={
    markings,
    mark=at position 0.64 with {\arrow{>}}}
    ] 
    \path[postaction={decorate},shorten >= -3pt, shorten <= -3pt] (A) to[out=0,in=270] (J);
\end{scope}

\begin{scope}[decoration={
    markings,
    mark=at position 0.755 with {\arrow{>}}}
    ] 
    \path[postaction={decorate},shorten >= -3pt, shorten <= -3pt] (A) to[out=0,in=270] (J);
\end{scope}

\begin{scope}[decoration={
    markings,
    mark=at position 0.87 with {\arrow{>}}}
    ] 
    \draw[postaction={decorate}] (A) to[out=0,in=270] (J);
\end{scope}

\draw[->,shorten >= -3pt, shorten <= -3pt] (I) to[bend right=3.5] (J);

\begin{scope}[decoration={
    markings,
    mark=at position 0.125 with {\arrow{>}}}
    ] 
    \path[postaction={decorate},shorten >= -3pt, shorten <= -3pt] (J) to[out=270,in=180] (S);
\end{scope}

\begin{scope}[decoration={
    markings,
    mark=at position 0.24 with {\arrow{>}}}
    ] 
    \path[postaction={decorate},shorten >= -3pt, shorten <= -3pt] (J) to[out=270,in=180] (S);
\end{scope}

\begin{scope}[decoration={
    markings,
    mark=at position 0.358 with {\arrow{>}}}
    ] 
    \path[postaction={decorate},shorten >= -3pt, shorten <= -3pt] (J) to[out=270,in=180] (S);
\end{scope}

\begin{scope}[decoration={
    markings,
    mark=at position 0.469 with {\arrow{>}}}
    ] 
    \path[postaction={decorate},shorten >= -3pt, shorten <= -3pt] (J) to[out=270,in=180] (S);
\end{scope}

\begin{scope}[decoration={
    markings,
    mark=at position 0.579 with {\arrow{>}}}
    ] 
    \path[postaction={decorate},shorten >= -3pt, shorten <= -3pt] (J) to[out=270,in=180] (S);
\end{scope}

\begin{scope}[decoration={
    markings,
    mark=at position 0.684 with {\arrow{>}}}
    ] 
    \path[postaction={decorate},shorten >= -3pt, shorten <= -3pt] (J) to[out=270,in=180] (S);
\end{scope}

\begin{scope}[decoration={
    markings,
    mark=at position 0.787 with {\arrow{>}}}
    ] 
    \path[postaction={decorate},shorten >= -3pt, shorten <= -3pt] (J) to[out=270,in=180] (S);
\end{scope}

\begin{scope}[decoration={
    markings,
    mark=at position 0.89 with {\arrow{>}}}
    ] 
    \draw[postaction={decorate}] (J) to[out=270,in=180] (S);
\end{scope}

\tikzset{>=latex}
 \path[line width=0pt] (A) -- (B) node[midway,above] {$u_0$};
 \path[line width=0pt] (B) -- (C) node[midway,above] {$t^{\delta_1}$};
 \path[line width=0pt] (C) -- (D) node[midway,above] {$u_1$};
 \path[line width=0pt] (D) -- (E) node[midway,sloped,above] {$\cdots$};
 \path[line width=0pt] (E) -- (F) node[midway,left] {$u_{k-m}$};
 \path[line width=0pt] (F) -- (G) node[midway,left] {$t^{\delta_{k-m+1}}$};
 \path[line width=0pt] (G) -- (H) node[midway,sloped,above] {$\cdots$};
 \path[line width=0pt] (H) -- (I) node[midway,left] {$t^{\delta_k}$};
 \path[line width=0pt] (I) -- (J) node[midway,left] {$u_k$};
 
 \path[line width=0pt] (J) -- (K) node[midway,right] {$v_0$};
 \path[line width=0pt] (K) -- (L) node[midway,right] {$t^{\varepsilon_1}$};
 \path[line width=0pt] (L) -- (M) node[midway,sloped,above] {$\cdots$};
 \path[line width=0pt] (M) -- (N) node[midway,right] {$t^{\varepsilon_m}$};
 \path[line width=0pt] (N) -- (O) node[midway,right] {$v_m$};
 \path[line width=0pt] (O) -- (P) node[midway,sloped,above] {$\cdots$};
 \path[line width=0pt] (P) -- (Q) node[midway,above=2pt] {$v_{\ell-1}$};
 \path[line width=0pt] (Q) -- (R) node[midway,above] {$t^{\varepsilon_\ell}$};
 \path[line width=0pt] (R) -- (S) node[midway,above] {$v_\ell$};
 
 \draw[->, shorten >= -1pt] (F) -- (N) node[midway,below] {$a$};
 
 \fill (A) circle (1pt);
 \fill (B) circle (1pt);
 \fill (C) circle (1pt);
 \fill (D) circle (1pt);
 \fill (E) circle (1pt);
 \fill (F) circle (1pt);
 \fill (G) circle (1pt);
 \fill (H) circle (1pt);
 \fill (I) circle (1pt);
 
 \fill (J) circle (1pt);
 
 \fill (K) circle (1pt);
 \fill (L) circle (1pt);
 \fill (M) circle (1pt);
 \fill (N) circle (1pt);
 \fill (O) circle (1pt);
 \fill (P) circle (1pt);
 \fill (Q) circle (1pt);
 \fill (R) circle (1pt);
 \fill (S) circle (1pt);
\end{tikzpicture}
\end{center}
\caption{\label{fig-peak-elimination} The situation from Lemma~\ref{red uv}.}
\end{figure}

\section{Knapsack semilinearity of HNN-extensions} 

The goal of this section is show that the HNN-extension 
\eqref{eq-HNN-C} is knapsack semilinear provided $G$ is knapsack semilinear relative to $\{1,A\}$.
We first introduce some concepts that can be found in a similar form in \cite{FiLoZe2021}.

We call a word $w\in \BR(H)$ {\em well-behaved}, if $w^m$ is Britton-reduced for every $m \geq 0$.
Note that $w$ is well-behaved if and only if $w$ and $w^2$ are Britton-reduced.
Note that every word $w \in \Sigma^*$ is well-behaved.

\begin{lem}[\mbox{c.f.~\cite[Lemma 6.3]{FiLoZe2021}}]\label{Britton-reduced-powers}
From a given word $u\in \BR(H)$ we can compute words $s,p,v \in \BR(H)$ such that $u^m =_H s v^m p$ for every $m\geq 0$ 
and $v$ is well-behaved.
\end{lem}
In the following we assume that $G$ is knapsack semilinear relative to $\{1,A\}$.
For a knapsack expression $e = v_0 u_1^{x_1} v_1 u_2^{x_2} v_2  \cdots u_k^{x_k} v_k$ over the alphabet $\Sigma \cup \{t,t^{-1}\}$
we define the knapsack expression 
\[ \pi_\Sigma(e) = \pi_\Sigma(v_0) \pi_\Sigma(u_1)^{x_1} \pi_\Sigma(v_1) \pi_\Sigma(u_2)^{x_2} \pi_\Sigma(v_2)  \cdots \pi_\Sigma(u_k)^{x_k} \pi_\Sigma(v_k)\]
over the alphabet $\Sigma$.

For an exponent expression $e(x_1,\ldots,x_n)$ over the alphabet $\Sigma \cup \{t,t^{-1}\}$ we call 
$$e(x_1,\ldots,x_n) \in_H G$$
a {\em $G$-constraint}. If $e$ is an exponent expression over the alphabet $\Sigma$, then $e \in_G A$ is called an $A$-constraint.
Since $G$ is knapsack semilinear relative to $A$, the set of solutions of an $A$-constraint is  semilinear.

\begin{lem} \label{2dim}
Let $u,v \in \BR(H)$ be well-behaved, $u'$ (resp., $v''$) be a proper prefix of $u$ (resp., $v$)
and $u''$ (resp., $v'$) be a proper suffix of  $u$ (resp., $v$).
Let $e=e(z_1,\ldots, z_k)$ be a knapsack expression over the alphabet $\Sigma$.
Then the set of all $(x,y,z_1,\ldots,z_k) \in \mathbb{N}^{k+2}$ such that the $G$-constraint
\begin{equation} \label{eq-xyz_i}
u'' u^x u' e(z_1,\ldots, z_k) \, v' v^y v''  \in_H G 
\end{equation}
holds is semilinear and a semilinear representation can be effectively computed from the words $u,v,u',u'',v',v'',e$.
\end{lem}

\begin{proof}
We first claim that by cyclically rotating $u$ and $v$ we can assume that $u'' = v'' = \varepsilon$. 
We only prove this for $u''$, for $v''$ we can argue analogously.
We can write $u = r u''$ for some word $r$. Then for all $x \in \N$ we have
$u'' u^x u' = u'' (r u'')^x u' = (u'' r)^x u'' u'$.
The word $u'' u'$ is either a prefix of $u'' r$ or we can write $u'' u' = (u'' r) \tilde{u}$ for some 
prefix $\tilde{u}$ of $u'' r$. In the first case, we can simply replace $u'' u^x u'$ in  \eqref{eq-xyz_i}
by  $(u'' r)^x u'' u'$ (note that $u'' r$ is well-behaved since it is a cyclic rotation of the well-behaved word $u = r u''$).
In the second case (where $u'' u' = (u'' r) \tilde{u}$ for some 
prefix $\tilde{u}$ of $u'' r$), we replace $u'' u^x u'$ in  \eqref{eq-xyz_i}
by  $(u'' r)^x \tilde{u}$.  If $L \subseteq \N^{\{x,y,z_1,\ldots,z_k\}}$ is the set of solutions of the resulting $G$-constraint, then
the formula $(x+1, y, z_1, \ldots, z_k) \in L$ describes the set of solution of \eqref{eq-xyz_i}. Clearly, a Presburger formula
for $L$ immediately yields a Presburger formula for $(x+1, y, z_1, \ldots, z_k) \in L$.

By the previous paragraph, it suffices to consider the set of solutions of the $G$-constraint
$$
u^x u' e(z_1,\ldots, z_k) \, v' v^y   \in_H G .
$$
This constraint holds (for certain values of $x,y,z_1,\ldots, z_k$) if and only if 
$u^x u' e(z_1,\ldots, z_k) \, v' v^y$ can be Britton-reduced to a word from $\Sigma^*$
which must be $\pi_\Sigma(u^x u' e(z_1,\ldots, z_k) \, v' v^y)$.
Since $u^x$ and $v^y$ are Britton-reduced for every $x,y \in \N$ we can apply Lemma~\ref{red uv}.

Let $S_u$ be the set of suffixes of $u$ that start with $t^{\pm 1}$ and let $P_{v}$ 
be the set of prefixes of $v$ that end with $t^{\pm 1}$. We define $S_{u'}$ and $P_{v'}$ 
analogously. Then by Lemma~\ref{red uv} the following formula is equivalent to $u^x u' e(z_1,\ldots,z_n) \, v' v^y  \in_H G$
(as usual, $\wedge$ denotes logical conjunction and $\Rightarrow$ denotes logical implication):
\begin{eqnarray*}
& & \!\!\!\!\!\!\!\!\!\!\!   \pi_t(u^x u') = \pi_t(v' v^y)^{-1} \ \wedge  \\
& & \!\!\!\!\!\!\!\!\!\!\!   \forall x' < x \, \forall y' < y :  \bigwedge_{s \in S_u} \bigwedge_{p \in P_v} \pi_t(s \, u^{x'} u') = \pi_t(v' v^{y'} p)^{-1} \ \Rightarrow \ \pi_\Sigma(s \, u^{x'} u' \, e \,  v' v^{y'}p) \in_G A \\
& & \!\!\!\!\!\!\!\!\!\!\!  \phantom{ \forall x' \leq x \, \forall y' \leq y :  } \bigwedge_{s \in S_u} \bigwedge_{p \in P_{v'}} \pi_t(s \, u^{x'} u') = \pi_t(p)^{-1} \ \Rightarrow \ \pi_\Sigma(s \, u^{x'} u' \, e \, p) \in_G A \\
& & \!\!\!\!\!\!\!\!\!\!\!  \phantom{ \forall x' \leq x \, \forall y' \leq y :  } \bigwedge_{s \in S_{u'}} \bigwedge_{p \in P_v} \pi_t(s) = \pi_t(v' v^{y'} p)^{-1} \ \Rightarrow \ \pi_\Sigma(s \, e \,  v' v^{y'}p) \in_G A \\
& & \!\!\!\!\!\!\!\!\!\!\!  \phantom{ \forall x' \leq x \, \forall y' \leq y :  } \bigwedge_{s \in S_{u'}} \bigwedge_{p \in P_{v'}} \pi_t(s) = \pi_t(p)^{-1} \ \Rightarrow \ \pi_\Sigma(s\, e \, p) \in_G A .
\end{eqnarray*}
Let us explain the intuition behind this formula; see also Figure~\ref{fig-eq-xyz_i} which shows a van Kampen diagram for 
$u^8 u' e(z_1,\ldots, z_k) \, v' v^5 =_H g$.

The formula $\pi_t(u^x u') = \pi_t(v' v^y)^{-1}$ expresses that every $t$ (resp., $t^{-1}$) in $u^x u'$ cancels with a $t^{-1}$ (resp., $t$) in $v' v^y$.
If this is not the case, then $u^x u' e(z_1,\ldots, z_k) \, v' v^y$ cannot be Britton-reduced to a word over $\Sigma$.
The other four lines of the formula ensure that the Britton-reduction of $u^x u' e(z_1,\ldots, z_k) \, v' v^y$ to a word over $\Sigma$ actually exists.
This Britton-reduction proceeds from right to left in Figure~\ref{fig-eq-xyz_i}. In each reduction step, the right-most slice in Figure~\ref{fig-eq-xyz_i} is eliminated.
Assume that the Britton-reduction has already eliminated the part to the right of the shaded slice.
The special form of our HNN-extension \eqref{eq-HNN-C} implies that if a word $w \in (\Sigma \cup \{t,t^{-1}\})^*$ is Britton-reduced
to a word over $\Sigma$, then we have $w =_H \pi_\Sigma(w)$ (every Britton-reduction step is of the form $t^{-1} a t \to a$ or 
$t a t^{-1} \to a$ for $a \in A$). Hence, we must have $a =_G \pi_\Sigma(s \, u^{4} u' \, e \,  v' v^{2}p)$ in Figure~\ref{fig-eq-xyz_i}.
In order to eliminate the shaded slice, the following must hold:
\begin{itemize}
\item $a$ must belong to the subgroup $A$, i.e., $\pi_\Sigma(s \, u^{4} u' \, e \,  v' v^{2}p) \in_G A$,
\item the first letter of $s$ (a suffix of $u$) must be $t$ (or $t^{-1}$), and
\item the last letter of $p$ (a prefix of $v$) must be $t^{-1}$ (or $t$); note that $v$ goes from right to left.
\end{itemize}
The second line in the above formula ensures that $\pi_\Sigma(s \, u^{x'} u' \, e \,  v' v^{y'}p) \in_G A$ whenever $x' < x$, $y' < y$,
$s$ is a suffix of $u$ that starts with $t^{\pm 1}$, $p$ is a prefix of $v$ that ends with $t^{\pm 1}$ and 
$\pi_t(s \, u^{x'} u') = \pi_t(v' v^{y'} p)^{-1}$ holds. This ensures that 
$s \, u^{x'} u' \, e \,  v' v^{y'}p$ is the group element represented by one of the vertical edges in Figure~\ref{fig-eq-xyz_i}.
The other parts of the above formula deal with the cases where the vertical edge has an endpoint in $u'$ or $v'$.
Altogether this ensures that all the vertical edges in Figure~\ref{fig-eq-xyz_i} represent elements of the subgroup $A$.

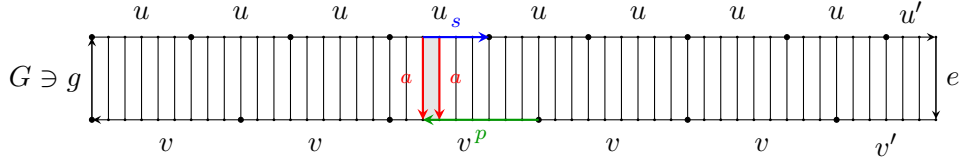
\begin{figure}[t]
\begin{tikzpicture}[, scale=1.1]
\node (A) at (0,0) {};
\node (B) at (0,1) {};
\node (C) at (10.2,0) {};
\node (D) at (10.2,1) {};
\draw[ -stealth, shorten <= -3.5, shorten >= -3.5] (A) -- (B) node[midway,left] {$G \ni g$};
\draw[ -stealth, shorten <= -3.5, shorten >= -3.5] (D) -- (C) node[midway,right] {$e$};
\draw[ -stealth, shorten <= -3.5, shorten >= -3.5] (C) -- (A)  {};
\draw[ -stealth, shorten <= -3.5, shorten >= -3.5] (B) -- (D)  {};

 \draw [fill=white!80!gray] (4,0) rectangle (4.2,1);
 
\foreach \x in {0,0.2,...,10.2}
{ \draw[fill] (\x,0) circle (.1mm);
  \draw[fill] (\x,1) circle (.1mm);
\draw (\x,1) -- (\x,0)  {};  
}
\foreach \x in {0,1.2,...,9.6}
{ \draw[fill] (\x,1) circle (.3mm);
}
\foreach \x in {0,1.8,3.6,5.4,7.2,9} 
{ \draw[fill] (\x,0) circle (.3mm);
   }
\foreach \x in {0,1.8,3.6,5.4,7.2} 
{  \node at (\x+.9,-.3) {$v$}; }
\foreach \x in {0,1.2,...,8.4}
{ \node at (\x+.6,1.3) {$u$}; }
 \node at (9.9,1.3) {$u'$}; 
 \node at (9.6,-.3) {$v'$}; 
 \draw[line width=0.28mm, red,-stealth] (4.2,1) -- (4.2,0)  node[midway,right=-0.08] {\scriptsize{$a$}}; 
 \draw[line width=0.28mm, red,-stealth] (4.0,1) -- (4.0,0) node[midway,left=-0.08] {\scriptsize{$a$}};   
 \draw[line width=0.28mm, blue, -stealth] (4.0,1) -- (4.8,1)  node[midway,above=-0.08] {\scriptsize{$s$}}; 
 \draw[line width=0.28mm, black!40!green, stealth-] (4.0,0) -- (5.4,0)  node[midway,below=-0.08] {\scriptsize{$p$}}; 

\end{tikzpicture}
\caption{The idea behind the proof of Lemma~\ref{2dim}. The fact that all vertical lines represent elements from the subgroup $A$ is expressed
by the formula from the proof.} \label{fig-eq-xyz_i}
\end{figure}

By Lemma~\ref{2dim-monoid} the solution set of the equation $\pi_t(u^x u') = \pi_t(v' v^y)^{-1}$  (which is interpreted over the free monoid $\{t,t^{-1}\}^*$)
 is semilinear. To see this let $w = v^{-1}$ and $w' = (v')^{-1}$. 
Then $\pi_t(u^x u') = \pi_t(v' v^y)^{-1}$ is equivalent to $\pi_t(u)^x \pi_t(u') = \pi_t(w)^y \pi_t(w') $. 
For the same reason, also the equation $\pi_t(s \, u^{x'} u') = \pi_t(v' v^{y'} p)^{-1}$ is equivalent to a semilinear constraint.
The solution sets of the equations $ \pi_t(s) = \pi_t(v' v^{y'} p)^{-1}$ and $\pi_t(s) = \pi_t(v' v^{y'} p)^{-1}$ are finite.
Moreover, each of the $A$-constraints
($\pi_\Sigma(s \, u^{x'} u' \, e \,  v' v^{y'}p) \in_G A$ etc.) is
equivalent to a semilinear constraint because $G$ is knapsack semilinear relative to $A$.
Hence, the above formula is equivalent to a Presburger formula and therefore defines a semilinear set.
\end{proof}

\begin{rem} \label{remark-1dim} 
There are variations of Lemma~\ref{2dim}, where the $G$-constraint has one of the following forms:
\begin{align*}
u' e(z_1,\ldots, z_k) \, v' v^y v''  & \in_H  G, \\
u' u^x u'' e(z_1,\ldots, z_k) \, v'  & \in_H  G, \\ 
\text{or } u' e(z_1,\ldots, z_k) \, v' & \in_H  G
\end{align*}
with $u,u',u'',v,v',v''$ as in Lemma~\ref{2dim}. In all cases, the set of solutions of the $G$-constraint can be shown to be effectively semilinear
using the arguments from the proof of Lemma~\ref{2dim}.
\end{rem}

\begin{defi} \label{def-reduction-HNN}
We define a reduction relation on tuples over $\BR(H)$ of arbitrary length.
Take $u_1, u_2, \ldots, u_m \in \BR(H)$ and $1 \leq i < j \leq m$. Then we have
\begin{align*}
& (u_1, u_2, \ldots, u_m)  \to  (u_1, \ldots, u_{i-1},\pi_\Sigma(u_{i}), u_{i+1}, \ldots, u_{j-1},  \pi_\Sigma(u_{j}), u_{j+1}, \ldots, u_m)
\end{align*}
if $\pi_t(u_i) \neq \varepsilon \neq \pi_t(u_j)$, $u_{i+1} \cdots u_{j-1} \in \Sigma^*$ and 
$u_i u_{i+1} \cdots u_{j-1} u_{j} \in_H G$. Note that this implies that $u_i u_{i+1} \cdots u_{j-1} u_{j}  =_H \pi_\Sigma(u_i) u_{i+1} \cdots u_{j-1} \pi_\Sigma(u_{j})$.
A concrete sequence of such rewrite steps leading to a tuple where all entries belong to $\Sigma^*$
is a {\em $G$-reduction} of the initial tuple, and the initial tuple is called {\em $G$-reducible}.
We also say that $u_i$ and $u_j$ {\em matched in a $G$-reduction}.
\end{defi}

A $G$-reduction of a tuple $(u_1, u_2, \ldots, u_m)$ can be seen as
a witness for the fact that $u_1 u_2 \cdots u_m \in_H G$. 
On the other hand, $u_1 u_2 \cdots u_m \in_H G$ does not necessarily imply
that $(u_1, u_2, \ldots, u_m)$ has a $G$-reduction.
But we can show that $u_1 u_2 \cdots u_m \in_H G$ implies that $(u_1, u_2, \ldots, u_m)$ can refined
(by factorizing the $u_i$) such that the resulting refined tuple has a $G$-reduction. Moreover,
it is important that we have an upper bound on the length of the refined tuple ($4m$ in Lemma~\ref{Britton-lemma-reduction} below)
that only depends on $m$  and not on the words
$u_1, u_2, \ldots, u_m$.

We say that the tuple 
$(v_1, v_2, \dots, v_n)$ is a \emph{refinement} of the tuple $(u_1, u_2, \dots, u_m)$ if there exist factorizations $u_i=u_{i,1} \cdots u_{i,k_i}$ in $(\Sigma \cup \{t,t^{-1}\})^*$ such that
$k_i=1$ whenever $u_i \in \Sigma^*$ and
$(v_1, v_2, \dots, v_n)=(u_{1,1}, \dots, u_{1,k_1}, \dots, u_{m,1}, \dots, u_{m,k_m})$.

\begin{lem} \label{Britton-lemma-reduction}
Let $m \geq 2$ and $u_1, u_2, \ldots, u_m \in \BR(H)$. If $u_1 u_2 \cdots u_m \in_H G$,
then there exists a $G$-reducible refinement of $(u_1, u_2, \dots , u_m)$ that has length at most
$4m$.
\end{lem} 

\begin{proof}
Let $\overline{u} = (u_1, u_2, \ldots, u_m)$.
Let us define $\gamma(\overline{u})$ as the number of pairs $(i,j)$ with $1 \leq i < j \leq m$ such that
$u_i u_{i+1} \cdots u_j$ is not Britton-reduced and $u_{i+1} \cdots u_{j-1} \in \Sigma^*$. Note that 
$\pi_t(u_i) \neq \varepsilon \neq \pi_t(u_j)$ for such a pair $(i,j)$.
Moreover, if we have two pairs $(i,j)$ and $(k, \ell)$ of this form, then either $j \leq k$ or $\ell \leq i$.
Let  $\tau(\overline{u})$ be the number of $i$ such that $\pi_t(u_i) \neq \varepsilon$.

We prove by induction over 
$\gamma(\overline{u}) + \tau(\overline{u})$ 
that there exists a $G$-reducible refinement of $\overline{u}$ that has length at most
$2\gamma(\overline{u}) + \tau(\overline{u})+ m \leq 4m$.

The case $m=2$ is trivial: either $\gamma(u_1, u_2) = \tau(u_1, u_2) = 0$ and $u_1, u_2 \in \Sigma^*$
or $\gamma(u_1, u_2) = 1$,  $\tau(u_1,u_2) = 2$ in which case $(u_1, u_2)$ must reduce in one step to $(\pi_\Sigma(u_1), \pi_\Sigma(u_2))$.
If $m \geq 3$ then $u_1 u_2 \cdots u_m$ must contain a pin. Since every $u_i$ is Britton-reduced, there must exist $i<j$
 such that $u_i u_{i+1} \cdots u_j$ is not Britton-reduced
and $u_{i+1} \cdots u_{j-1} \in \Sigma^*$.
By Lemma~\ref{red uv} we can factorize $u_i$ and $u_j$ in
$(\Sigma \cup \{t,t^{-1}\})^*$ as
$u_i=u'_i r$ and $u_j = s u'_j$ such that $r s \in_H G$
and $u_i u_{i+1} \cdots u_{j-1} u_j =_H  u'_i \pi_\Sigma(r) u_{i+1} \cdots u_{j-1} \pi_\Sigma(s) u'_j$ is Britton-reduced.
Note that $r$ and $s$ must contain $t$ or $t^{-1}$. Moreover, we can assume that 
either $u'_i = \varepsilon$ or $u'_i$ ends with $t^{\pm 1}$ (if $u'_i$ ends with a non-empty word over $\Sigma$, we 
can remove this word from $u'_i$ and add it to $r$)
and, similarly, either $u'_j = \varepsilon$ or $u'_j$ begins with $t^{\pm 1}$.

\medskip
\noindent
{\em Case 1.}  $u'_i$ and $u'_j$ both contain $t^{\pm 1}$.
Then we have
$$
\gamma(u_1, \ldots, u_{i-1}, u'_i, \pi_\Sigma(r), u_{i+1}, \ldots, u_{j-1}, \pi_\Sigma(s), u'_j,u_{j+1}, \ldots, u_m) <
\gamma(\overline{u}) 
$$
since $u'_i \pi_\Sigma(r) u_{i+1} \cdots u_{j-1} \pi_\Sigma(s) u'_j$ is Britton-reduced and $u'_i$ and $u'_j$ both contain $t^{\pm 1}$.
Moreover, we have
$$
\tau(u_1, \ldots, u_{i-1}, u'_i, \pi_\Sigma(r), u_{i+1}, \ldots, u_{j-1}, \pi_\Sigma(s), u'_j,u_{j+1}, \ldots, u_m) =
\tau(\overline{u}) .
$$
Hence, we can apply the induction hypothesis to the tuple 
\begin{equation}\label{eq-tuple1}
(u_1, \ldots, u_{i-1}, u'_i, \pi_\Sigma(r), u_{i+1}, \ldots, u_{j-1}, \pi_\Sigma(s), u'_j,u_{j+1}, \ldots, u_m) .
\end{equation}
It must have a $G$-reducible refinement of length at most
$$
2 (\gamma(\overline{u}) - 1) + \tau(\overline{u}) + m+2 = 2\gamma(\overline{u}) + \tau(\overline{u}) + m.
$$ 
In this refinement $\pi_\Sigma(r), \pi_\Sigma(s) \in \Sigma^*$ will not be
factorized into more than one factor. We therefore can take the refinement of \eqref{eq-tuple1} and replace $\pi_\Sigma(r)$ and $\pi_\Sigma(s)$ by $r$ and $s$, respectively.
This leads to  a $G$-reducible of our original tuple $\overline{u}$ having length at most
$2\gamma(\overline{u}) + \tau(\overline{u})+ m$.

\medskip
\noindent
{\em Case 2.} $u'_i = \varepsilon$ and $u'_j$ begins with $t^{\pm 1}$. Then we have
$$
\gamma(u_1, \ldots, u_{i-1}, \pi_\Sigma(u_i), u_{i+1}, \ldots, u_{j-1}, \pi_\Sigma(s), u'_j,u_{j+1}, \ldots, u_m) \leq
\gamma(\overline{u}) 
$$
and 
$$
\tau(u_1, \ldots, u_{i-1}, \pi_\Sigma(u_i), u_{i+1}, \ldots, u_{j-1}, \pi_\Sigma(s), u'_j,u_{j+1}, \ldots, u_m) <
\tau(\overline{u}) .
$$
We can therefore apply the induction hypothesis to the tuple
\begin{equation}\label{eq-tuple2}
(u_1, \ldots, u_{i-1}, \pi_\Sigma(u_i), u_{i+1}, \ldots, u_{j-1}, \pi_\Sigma(s), u'_j,u_{j+1}, \ldots, u_m)
\end{equation}
and obtain a $G$-reducible refinement of length at most 
$$
2 \gamma(\overline{u}) + \tau(\overline{u}) - 1 + m+1 = 2\gamma(\overline{u}) + \tau(\overline{u})+ m.
$$
Replacing $\pi_\Sigma(u_i)$ by $u_i$ and $\pi_\Sigma(s)$ by $s$ in this refinement yields a $G$-reducible refinement of $\overline{u}$.

The remaining cases where  (i) $u'_j = \varepsilon$ and $u'_i$ ends with $t^{\pm 1}$ or (ii) $u'_i = u'_j = \varepsilon$ are analogous to case 2.
This concludes the proof of the lemma.
\end{proof}
Now we are able to prove the main theorem of this section.

\begin{thm}\label{HNNsemilin}
Let $H=\langle G, t \mid t^{-1}at=a~(a\in A)\rangle$ be an HNN-extension, where $G$ is 
knapsack semilinear relative to $\{1,A\}$.
Then $H$ is knapsack semilinear.
\end{thm}

\begin{proof}
The proof of the theorem is based on  ideas from \cite{FiLoZe2021}.
Consider a knapsack expression 
$$
e(x_2,x_4,\ldots,x_m) = u_1 u_2^{x_2} u_3 u_4^{x_4} u_5 \cdots u_{m}^{x_{m}} u_{m+1} 
$$
with $m$ even (later it will convenient to have only variables with an even index).
We can assume that all $u_i$ are Britton reduced. Moreover,
by Lemma~\ref{Britton-reduced-powers}, we can assume that every $u_{i}$ with $i$ even is well-behaved
and moreover non-empty (otherwise we can remove the power $u_i^{x_i}$).

In the following we describe an algorithm that computes a semilinear representation of $\Sol_{H}(e)$ in three main steps.
The algorithm transforms logical statements into equivalent logical statements
(we do not have to define the precise logical language; the meaning of the statements should be always clear).
Every statement contains the variables $x_2, x_4, \ldots, x_{m}$ from our knapsack expression and equivalence of two statements means
that for every valuation $\sigma : \{x_2, x_4, \ldots, x_{m}\} \to \N$ the two statements yield the same truth value.
We start with the statement $e(x_2,x_4,\ldots,x_{m}) =_H 1$ and end with a Presburger formula. In each of the three steps we transform the current
statement $\Phi$ into an equivalent disjunction $\bigvee_{i=1}^n \Phi_i$. 
We can therefore view the whole process
as a branching tree of depth three, where the nodes are labelled with statements. If a node is labelled with $\Phi$ and its children are labelled
with $\Phi_1, \ldots, \Phi_n$ then $\Phi$ is equivalent to $\bigvee_{i=1}^n \Phi_i$. The leaves of the tree are labelled with 
Presburger formulas with free variables $x_2, x_4, \ldots, x_{m}$. 
We will concentrate on a single branch of this tree, which can be viewed as a sequence of nondeterministic guesses. 

Let $N_\Sigma \subseteq [1,m+1]$ be the set of indices such that $u_i \in \Sigma^*$ and 
let $N_t = [1,m+1] \setminus N_\Sigma$ be the set of indices such that $\pi_t(u_i) \neq \varepsilon$. 
Moreover, let us define $w_{i} = u_{i}$ for $i$ odd and
$w_{i} = u_{i}^{x_i}$ for $i$ even.

By Lemma~\ref{lemma-id-in-H}, $e =_H 1$ is equivalent to
$e \in_H G \, \wedge \, \pi_\Sigma(e) =_G 1$.
Since $G$ is knapsack semilinear, the set $\Sol_G(\pi_\Sigma(e))$ is semilinear. Hence,
it suffices to show that the set of all $(x_2,x_4,\ldots, x_{m}) \in \N^{m/2}$ with 
$e(x_2,x_4\ldots,x_{m}) \in_H G$ is semilinear. Here, we will use the assumption that
$G$ is knapsack semilinear relative to $A$.

\medskip
\noindent
{\em Step 1: Applying Lemma~\ref{Britton-lemma-reduction}.}  
We construct a disjunction $\Psi$ from the knapsack expression $e$ using  Lemma~\ref{Britton-lemma-reduction}. More precisely, we construct 
$\Psi$ by nondeterministically guessing the following data:
\begin{enumerate}[(i)]
\item symbolic factorizations $w_i = y_{i,1} \cdots y_{i,k_i}$ in 
$(\Sigma \cup \{t,t^{-1}\})^*$ for all $i \in [1,m+1]$.
Here the $y_{i,j}$ are existentially quantified variables that take values 
in $\BR(H)$. Later, these variables will be eliminated. The guessed $k_i$ must satisfy
$k_i \geq 1$ for all $i$, 
$k_i=1$ for all $i \in N_\Sigma$, 
and $\sum_{1 \le i \le m+1} k_i \leq 4(m+1)$. 
\item a $G$-reduction (according to Definition~\ref{def-reduction-HNN}) of the tuple
$$
(y_{1,1} \cdots y_{1,k_1}, \ldots, y_{m+1,1} \cdots y_{m+1,k_{m+1}}).
$$
\end{enumerate}
For the $w_i$ with $i$ odd we can of course guess concrete words for the variables $y_{i,1}, \ldots, y_{i,k_i}$.
Later, we will do this, but in order to simplify the notation, we will still use the names $y_{i,1}, \ldots, y_{i,k_i}$ for these words.

For every specific guess in (i) and (ii)
we write down the conjunction of the following formulas:
\begin{itemize}
\item the equation $w_i = y_{i,1} \cdots y_{i,k_i}$ from (i) (every variable $y_{i,j}$ is existentially quantified) and
\item all $G$-constraints that result from $G$-reduction steps in the guessed $G$-reduction (this will made more precise in Step 2 below).
\end{itemize}
The formula $\Psi$ is the disjunction of the above existentially quantified conjunctions, taken over all possible guesses in (i) and (ii).
This formula is equivalent to the $G$-constraint $e \in_H G$.

\medskip
\noindent
{\em Step 2: Eliminating the equations $w_{i} = y_{i,1} \cdots y_{i,k_i}$.} 
For an odd $i$ (i.e., $w_i = u_i$) we can eliminate this equation by guessing
a concrete factorization $u_i = u_{i,1} \cdots u_{i,k_i}$ and then replace the 
equation $w_{i} = y_{i,1} \cdots y_{i,k_i}$ by the conjunction
$$
\bigwedge_{j=1}^{k_i} y_{i,j} = u_{i,j} .
$$
For an even $i$ (i.e., $w_i = u_i^{x_i}$) we can eliminate the equation $w_{i} = y_{i,1} \cdots y_{i,k_i}$
by guessing a symbolic factorization of $u_i^{x_i}$ into $k_i$ factors.
A specific guess leads to a formula
\begin{equation} \label{eq-y_ij}
\bigwedge_{j=1}^{k_i} y_{i,j} = u''_{i,j} u_i^{x_{i,j}} u'_{i,j+1} \ \wedge \ x_i = c_i + \sum_{j=1}^{k_i} x_{i,j} .
\end{equation}
Here, every $u'_{i,j}$ ($2 \le j \le k_i$) is a proper prefix of $u_i$ and every $u''_{i,j}$ ($2 \le j \le k_i$) is a proper suffix of $u_i$ such that either
$u_{i} = u'_{i,j} u''_{i,j}$ or $u'_{i,j} = u''_{i,j}=\varepsilon$ for all $2 \leq  j \leq k_i$. We set $u'_{i,k_i+1} = u''_{i,1} =\varepsilon$ in the above formula.
Moreover, $c_i$ is the number of $2 \leq j \leq k_i$ for which $u'_{i,j} \neq \varepsilon \neq u''_{i,j}$ holds.
The $u'_{i,j}$ and $u''_{i,j}$ are nondeterministically guessed.

We also guess which of the new exponent variables $x_{i,j}$ are zero and which of the $x_{i,j}$ are non-zero.
If we guess $x_{i,j}=0$, then we replace $x_{i,j}$ in \eqref{eq-y_ij} by $0$. This yields the 
equation $y_{i,j} = u''_{i,j} u'_{i,j+1}$. If we guess $x_{i,j}>0$, then we add this constraint to \eqref{eq-y_ij}.
After this step, it is determined whether a $y_{i,j}$ contains $t$ or $t^{-1}$ (for $i$ even as well as for $i$ odd). Those $y_{i,j}$ must be matched
by $G$-reduction steps in the $G$-reduction that we guessed in Step 1. In fact, what we guessed in Step 1 is such a matching.

\medskip
\noindent
{\em Step 3: Eliminating $G$-constraints.}  Assume that $y_{i,j}$ and $y_{k,\ell}$ are matched in the guessed $G$-reduction.
W.l.o.g. assume that $i<k$ or $i=k$ and $j < \ell$, i.e., $(i,j)$ is lexicographically before $(k,\ell)$.
Then our formula contains the $G$-constraint
$$
y_{i,j}  \left(\prod_{(i,j) \prec (p,q) \prec(k,\ell)} \pi_\Sigma(y_{p,q})\right) y_{k,\ell} \in_H G ,
$$
where $\prec$ is the strict lexicographic order on pairs of natural numbers. 
In this constraint, we can replace every $y_{a,b}$ with $a$ even by $u''_{i,j} u_i^{x_{i,j}} u'_{i,j+1}$
(or $u''_{a,b} u'_{a,b+1}$ in case $x_{i,j}=0$ was guessed), whereas every $y_{a,b}$ with $a$ odd can
be replaced by the concrete word $u_{a,b}$.
If both  $y_{i,j}$ and $y_{k,\ell}$ contain an exponent variable we obtain 
a $G$-constraint of the form \eqref{eq-xyz_i}. If $y_{i,j}$ or $y_{k,\ell}$ is a concrete word we obtain 
a $G$-constraint having one of the three forms listed in Remark~\ref{remark-1dim}. 
Lemma~\ref{2dim} and Remark~\ref{remark-1dim} imply that in each case,  
the set of solutions of the $G$-constraint is semilinear.
This concludes the proof of the theorem.
\end{proof}

For a subset $S \subseteq G$ of the group $G$ one defines the {\em centralizer} 
$$C(S)=\{h\in G \mid gh=hg \text{ for all } g \in S \}.
$$
The HNN-extension $H = \langle G,t \mid t^{-1} a t = a \,  (a \in C(S)) \rangle$ is an \emph{extension of the centralizer $C(S)$}.
Extensions of centralizers were first studied in \cite{MyasnikovR96} in the context of exponential groups. 

\begin{thm} \label{thm-ext-centralizer}
If $G$ is knapsack semilinear and $H$ is an extension of a centralizer $C(S)$ with $S$ finite, then $H$ is knapsack semilinear as well.
\end{thm}

\begin{proof}
We have to show that $G$ is also knapsack semilinear relative to $C(S)$.
Let $e = e(x_1,\ldots,x_n)$ be a knapsack expression. Then $e \in_G C(S)$ is equivalent to $\bigwedge_{a \in S} e a =_G a e$.
Note that $e a =_G a e$ is equivalent to $e a e^{-1} a^{-1} =_G1$ and $e a e^{-1} a^{-1}$ is an exponent expression.
Since $G$ is knapsack semilinear and semilinear sets are closed under finite intersections, the set of solutions of $\bigwedge_{a \in S} e a = a e$
is semilinear.
\end{proof}

\begin{rem} 
It is straightforward to generalize Theorem~\ref{HNNsemilin} to a multiple HNN-extension
$$
H=\langle G, t_1, \ldots, t_n \mid t_i^{-1}at_i=a~(a\in A_i, 1 \leq i \leq n)\rangle .
$$
If $G$ is 
knapsack semilinear relative to $\{1,A_1, \ldots, A_n\}$ then $H$ is knapsack semilinear.
\end{rem}

\section{Quasi-convex subgroups of hyperbolic groups}

In this section we show that hyperbolic groups are knapsack semilinear relative to quasiconvex subgroups.
We start with the definition hyperbolic groups.

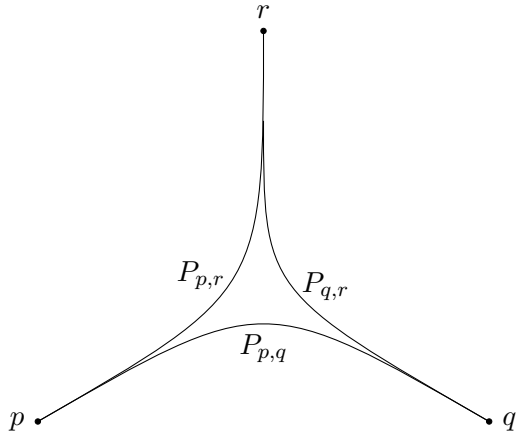
\begin{figure}[t]
 \centering{
 \scalebox{1}{
\begin{tikzpicture}
\tikzstyle{small} = [circle,draw=black,fill=black,inner sep=.25mm]
\node (p) at (0,0) [small, label=left:$p$] {};
\node (q) at (6,0) [small, label=right:$q$] {};
\node (r) at (3,5.19615) [small, label=above:$r$] {};
\draw (0,0) .. controls (3,1.73205)  .. node[pos=.5,below=0mm] {$P_{p,q}$} (6,0);
\draw (0,0) .. controls (3,1.73205)  .. node[pos=.5,above=0mm, left=0mm] {$P_{p,r}$} (3,5.19615);
\draw (3,4) .. controls (3,1.73205)  .. node[pos=.5,above=0mm, right=0mm] {$P_{q,r}$} (6,0);
\end{tikzpicture}
}}
\caption{\label{fig-geo-tri}The shape of a geodesic triangle in a hyperbolic group}
\end{figure}

\subsection{Cayley-graphs and hyperbolic groups} \label{sec-hyp}

Let $G$ be a finitely generated group with the finite symmetric generating set $\Sigma$ and let $h : \Sigma^* \to G$
be the evaluation morphism.
The  {\em Cayley-graph} of $G$ (with respect to $\Sigma$) is the graph $\Gamma = \Gamma(G)$ with node set
$G$ and all edges $(g,ga)$ for $g \in G$ and $a \in \Sigma$. We view $\Gamma$ as a geodesic metric space,
where every edge $(g,ga)$ is identified with a unit-length interval. 
It is convenient to label the directed edge
from $g$ to $ga$ with the generator $a$. Note that since $\Sigma$ is symmetric,
there is also an edge from $ga$ to $g$ labelled with $a^{-1}$.
Therefore one can view $\Gamma$ as an undirected graph.
The distance between two points $p,q$ is denoted with $d_\Gamma(p,q)$.
For $g \in G$ let $|g| = d_\Gamma(1,g)$. For $\kappa \ge 0$ and $g \in G$ 
let $\mathcal{B}_\kappa(g) = \{ h \in G \mid d_\Gamma(g,h) \le \kappa\}$ be the ball of radius $\kappa$ around $g$.

Paths can be defined in a very general way for metric spaces, but we only need paths
that are induced by words over $\Sigma$.
Given a word $w = a_1 a_2 \cdots a_n$ (with $a_i \in \Sigma$), one obtains a unique path $P[w] : [0,n] \to \Gamma$,
which is a continuous mapping from the real interval $[0,n]$ to $\Gamma$.
It maps the subinterval $[i,i+1] \subseteq [0,n]$  isometrically 
onto the edge $(h(a_1 \cdots a_i), h(a_1\cdots a_{i+1}))$ of $\Gamma$.
The path $P[w]$ starts in $1$ and ends in $h(w)$ (the group element represented by $w$).
We also say that $P[w]$ is the unique path that starts in $1$ and is labelled with the word $w$.
More generally, for $g \in G$ we denote
with $g \cdot P[w]$ the path that starts in $g$  and is labelled with $w$.
When writing $u \cdot P[w]$ for a word $u \in \Sigma^*$, we mean the path
$h(u) \cdot P[w]$.
A path  $P : [0,n] \to \Gamma$ of the above form is {\em geodesic} if $d_\Gamma(P(0),P(n)) = n$ and it is
a $(\lambda,\epsilon)$-{\em quasigeodesic} if for all points
$p = P(a)$ and $q = P(b)$ we have $|a-b| \leq \lambda \cdot d_\Gamma(p,q) + \varepsilon$.

A word $w \in \Sigma^*$ is geodesic if the path $P[w]$ is geodesic, which means
that there is no shorter word representing the same group element from $G$.
Similarly, we define the notion of 
$(\lambda,\epsilon)$-quasigeodesic words.
A set (or language) of words $L \subseteq \Sigma^*$ is called geodesic (resp., $(\lambda,\epsilon)$-quasigeodesic), if every $w \in L$ is geodesic
(resp., $(\lambda,\epsilon)$-quasigeodesic).

A {\em geodesic triangle} consists of three points $p,q,r \in G$ and geodesic paths $P_1 = P_{p,q}$, $P_2 = P_{p,r}$, $P_3 = P_{q,r}$
(the three sides of the triangle),
where $P_{x,y}$ is a geodesic path from $x$ to $y$. We call a geodesic triangle {\em $\delta$-slim}
for $\delta \geq 0$, if for all $i \in \{1,2,3\}$, every point on $P_i$ has distance at most 
$\delta$ from a point on $P_j \cup P_k$, where $\{j,k\} = \{1,2,3\} \setminus \{i\}$.
The group $G$ is called  {\em $\delta$-hyperbolic}, if  every geodesic triangle is 
$\delta$-slim. Finally, $G$ is  hyperbolic, if it is  $\delta$-hyperbolic
for some $\delta \geq 0$. Figure~\ref{fig-geo-tri} shows the shape of a geodesic triangle in a hyperbolic group.
Finitely generated free groups are for instance $0$-hyperbolic with respect to a free finite generating set.
The property of being hyperbolic is independent of the chosen generating set $\Sigma$. 
The word problem for every hyperbolic group can be decided in real time \cite{Hol00}.

\subsection{Asynchronous biautomatic structures}

Let $G$ be a f.g.~group with the finite symmetric generating set $\Sigma$ and let 
$h : \Sigma^* \to G$ be the evaluation morphism. 
An {\em asynchronous biautomatic structure} for $G$ consists of a regular language $L \subseteq \Sigma^*$ 
such that the following holds; see also \cite{EpsCHLPT92,NeuSha94}:
\begin{itemize}
\item $G = h(L)$,
\item the relation $\{ (u,v) \in L \times L \mid u =_G v \}$ is rational, and
\item for every generator $a \in \Sigma$ the relations 
\[
\{ (u,v) \in L \times L \mid ua =_G v \} \text{ and }
\{ (u,v) \in L \times L \mid au =_G v \}
\]
are rational.
\end{itemize}
If in the last point it is only required that the relation $\{ (u,v) \in L \times L \mid ua =_G v  \}$ is rational,
then $L$ is called an asynchronous automatic structure for $G$.
A f.g.~group $G$ is called asynchronously (bi)automatic if it has an asynchronous (bi)automatic structure.
We need the following lemma.

\begin{lem}  \label{lemma-asynchronous-biautomatic}
Let $L$ be an asynchronous biautomatic structure for $G$, let $L_1$ and $L_2$  be regular subsets of $L$ and
let $v_1, v_2 \in \Sigma^*$. Then the relation
$$\{  (u_1,u_2) \in L_1 \times L_2 \mid v_1 u_1 =_G u_2 v_2 \}
$$
is rational.
Moreover, a finite state transducer for this relation can be effectively computed from the words $v_1, v_2$ and 
finite automata for $L_1$ and $L_2$.
\end{lem}

\begin{proof}
It suffices to show that the relation $R := \{  (u_1,u_2) \in L \times L \mid v_1 u_1 =_G u_2 v_2 \}$
is rational. The corresponding finite state transducer can in addition simulate the automaton for $L_1$ (resp., $L_2$) 
on the first (resp., second) tape. Rationality of the relation 
$R$ can be shown by induction on $|v_1|+|v_2|$.
The case $v_1 = v_2 = \varepsilon$ is clear. Assume w.l.o.g.~that $v_1 \neq \varepsilon$ and let $v_1 = v'_1 a$ 
with $a \in \Sigma$. By induction, the relation 
$R_1 = \{  (u'_1,u_2) \in L \times L \mid v'_1 u'_1 =_G u_2 v_2 \}$
is rational. Moreover, the relation 
$R_2 = \{  (u_1,u'_1) \in L \times L \mid a u_1 =_G u'_1 \}$
is rational as well. Finally, we have
$R = R_2 \circ R_1$,
where $\circ$ is relational composition.
The lemma follows since the class of rational relations is closed under relational composition \cite{SiTe00}. 
\end{proof}
We also need the following result from \cite{HoltRees03}:

\begin{lem}
Let $G$ be a hyperbolic group and let $\Sigma$ be a finite symmetric generating set for $G$.
Let $\lambda$ and $\epsilon$ be fixed constants. Then the set of all $(\lambda,\epsilon)$-quasigeodesic
words over the alphabet $\Sigma$ is an asynchronous biautomatic structure for $G$.
\end{lem}
In \cite{HoltRees03} it is only stated that the set of all $(\lambda,\epsilon)$-quasigeodesic words is an 
asynchronous automatic structure for $G$. But since for every $(\lambda,\epsilon)$-quasigeodesic 
word $w \in \Sigma^*$ also $w^{-1}$ is $(\lambda,\epsilon)$-quasigeodesic, it follows easily that the 
set of all $(\lambda,\epsilon)$-quasigeodesic words is an asynchronous biautomatic structure for $G$.
With Lemma~\ref{lemma-asynchronous-biautomatic} we obtain the following lemma.

\begin{lem} \label{lemma-rat-rel}
Let $G$ be a hyperbolic group with the finite symmetric generating set $\Sigma$ and let $\lambda$ and $\epsilon$ be fixed constants. 
Assume that $L_1, L_2 \subseteq \Sigma^*$ are $(\lambda,\epsilon)$-quasigeodesic
regular languages and $v_1, v_2 \in \Sigma^*$.
Then the relation
$$\{  (u_1,u_2) \in L_1 \times L_2 \mid v_1 u_1 =_G u_2 v_2 \}$$ is rational.
Moreover, a finite state transducer for this relation can be effectively computed from the words $v_1, v_2$ and 
finite automata for $L_1$ and $L_2$.
\end{lem}

\subsection{Parikh images in hyperbolic groups}

Let us fix a hyperbolic group $G$ with the finite symmetric generating set $\Sigma$
for the rest of the section.
We fix an arbitrary enumeration $a_1, \ldots, a_k$ of the alphabet $\Sigma$ in order to make Parikh images well-defined.
Recall that the semilinear sets are exactly the Parikh images of regular languages; see
Theorem~\ref{thm-Parikh-regular}. Together with  Lemma~\ref{lemma-rat-rel} we obtain the next result.

\begin{lem} \label{lemma-k=2-reg}
Let $G$ be a hyperbolic group with the finite symmetric generating set $\Sigma$ and let $\lambda$ and $\epsilon$ be fixed constants. 
Assume that $L_1, L_2 \subseteq \Sigma^*$ are $(\lambda,\epsilon)$-quasigeodesic
regular languages and $v_1, v_2 \in \Sigma^*$.
Then the set 
\begin{equation} \label{eq-lemma-k=2-reg}
\{  (P(u_1),P(u_2)) \in \mathbb{N}^{2k} \mid u_1 \in L_1, u_2 \in L_2, v_1 u_1 =_G u_2 v_2 \}
\end{equation}
is semilinear.
Moreover, a semilinear representation for this set can be effectively computed from the words $v_1, v_2$ and 
finite automata for $L_1$ and $L_2$.
\end{lem}

\begin{proof}
Let $\Sigma' = \{ a' \mid a \in \Sigma\}$ be a disjoint copy of the alphabet $\Sigma$.
By Lemma~\ref{lemma-rat-rel} there is a finite state transducer $\mathcal{T}$ for the relation
$$\{  (u_1,u_2) \in L_1 \times L_2 \mid v_1 u_1 =_G u_2 v_2 \}.
$$
From $\mathcal{T}$ we obtain a finite automaton $\mathcal{A}$ over the alphabet $\Sigma \cup \Sigma'$ by replacing every
transition $(p, a,\varepsilon, q)$ by $(p, a, q)$ and every
transition $(p, \varepsilon,a, q)$ by $(p, a', q)$.
For the alphabet $\Sigma \cup \Sigma'$ we take the enumeration $a_1, \ldots, a_k, a'_1, \ldots, a'_k$.
With this enumeration, the set \eqref{eq-lemma-k=2-reg}
is the Parikh image of the language $L(\mathcal{A})$. Hence, the lemma follows from Theorem~\ref{thm-Parikh-regular}.
\end{proof}

\subsection{The main result}
We now come to the main technical  result of this section.

\begin{thm} \label{theorem-hyp-main0}
Let $G$ be a hyperbolic group with the finite symmetric generating set $\Sigma=\{a_1,\ldots,a_k\}$.
Fix constants $\epsilon, \lambda$.
For $1\leq i\leq n$ 
let $L_i \subseteq \Sigma^*$ be a regular $(\lambda,\epsilon)$-quasigeodesic  language.
Then the set
$$
\{ (P(w_1), \ldots, P(w_n)) \in \mathbb{N}^{nk} \mid w_i \in L_i \text{ for } 1 \leq i \leq n, w_1 w_2 \cdots w_n =_G 1  \}
$$ 
is semilinear and a semilinear representation of this set can be computed from 
 finite automata for $L_1, \ldots, L_n$.
\end{thm}
We postpone the proof of Theorem~\ref{theorem-hyp-main0} and first derive some corollaries.

\begin{thm} \label{theorem-knapsack-semilinear-rel.A}
Let $G$ be hyperbolic and let $S \subseteq \Sigma^*$ be a regular geodesic set.
Then $G$ is knapsack semilinear relative to $h(S)$, where $h : \Sigma^* \to G$ is the evaluation morphism.
\end{thm}

\begin{proof}
Let $e$ be a knapsack expression over $\Sigma$.
We want to find a semilinear representation for the set 
\begin{equation} \label{eq-set1}
\{ \sigma : X_e \to \N \mid \exists w \in S : \sigma(e) =_G w \}  =
\bigcup_{w \in S} \Sol_G(ew^{-1}) .
\end{equation}
In a first step we show that we can assume that for every power $u^x$ that appears in $e$ ($x \in X_e$)
the language $u^*$ is $(\lambda,\epsilon)$-quasi\-geodesic 
for fixed constants $\lambda,\epsilon$ that only depend on the group $G$.
For this we use \cite[Proposition~8.4]{LOHREY2019}. It states that from our given knapsack expression
$e$ one can compute a finite list of knapsack expressions $e_1, \ldots, e_n$ with $X_{e_i} \subseteq X_e$ for $1 \le i \le n$,
functions $m_i, d_i : X_{e_i} \to \N$ and semilinear sets $\mathcal{F}_i \subseteq \N^{X_e \setminus X_{e_i}}$ 
such that 
\begin{equation} \label{eq-prop8.4}
\Sol_G(e) = \bigcup_{1 \le i \le n} (m_i \cdot \Sol_G(e_i) + d_i) \oplus \mathcal{F}_i .
\end{equation}
Moreover, every knapsack expression $e_i$ has the property that for every power $u^{x}$ that appears
in $e_i$, the language $u^*$ is $(\lambda,\epsilon)$-quasigeodesic for some fixed constants 
$\lambda,\epsilon$ that only depend on the group $G$.

Consider now for an arbitrary word $w \in S$ the knapsack expression $ew^{-1}$. From the construction in 
\cite[Proposition~8.4]{LOHREY2019}, it follows that 
\begin{equation} \label{eq-prop8.4'}
\Sol_G(ew^{-1}) = \bigcup_{1 \le i \le n}  (m_i \cdot \Sol_G(e_iw^{-1}) + d_i) \oplus \mathcal{F}_i .
\end{equation}
The reason is that in the proof of \cite[Proposition~8.4]{LOHREY2019},
every $e_i$ is computed from $e$ by replacing every power $u^{x}$ in $e$ either by a fixed power $u^k$ for some 
$k \in \mathbb{N}$ or by $v {\tilde{u}}^{x} v'$ for a $(\lambda,\epsilon)$-quasigeodesic  word $\tilde{u}$
and words $v$ and $v'$. The same replacements are also made for $ew^{-1}$.
Therefore, the set \eqref{eq-set1} is equal to
\begin{eqnarray*}
\bigcup_{w \in S} \Sol_G(ew^{-1})    
& = & \bigcup_{w \in S} \bigcup_{1 \le i \le n}  (m_i \cdot \Sol_G(e_iw^{-1}) + d_i) \oplus \mathcal{F}_i \\
& = & \bigcup_{1 \le i \le n}  \left( m_i \cdot  \left( \bigcup_{w \in S} \Sol_G(e_iw^{-1}) \right) + d_i \right) \oplus \mathcal{F}_i .
\end{eqnarray*}
The closure properties of semilinear sets imply that the set \eqref{eq-set1}  is semilinear provided 
 the set 
$$
\bigcup_{w \in S} \Sol_G(e_iw^{-1}) = \{ \sigma  : X_{e_i} \to \N \mid \exists w \in S : \sigma(e_i) =_G w\}
$$
 is semilinear for all $1 \le i \le n$.

This shows that it suffices to find a semilinear representation of \eqref{eq-set1} for a knapsack expression
$e = v_0 u_1^{x_1} v_1  u_1^{x_1} \cdots  u_n^{x_n} v_n$ where
all $u_i$ have the property that $u_i^*$ is a regular $(\lambda,\epsilon)$-quasigeodesic language.
Clearly, we can also assume that every $u_i$ is non-empty and every $v_i$ is geodesic. Moreover, since $S$ is regular and geodesic, it is easy
to see that also $S^{-1}$ is regular and geodesic.

Let $(L_1, \ldots, L_m)$ be the tuple of languages $(\{v_0\}, u_1^*,\{v_1\}, \ldots, u_n^*,\{v_n\}, S^{-1})$ (with $m = 2(n+1)$).
All these languages are regular and $(\lambda,\epsilon)$-quasigeodesic.
By Theorem~\ref{theorem-hyp-main0}, the set 
$$
\{ (P(w_1), \ldots, P(w_{m})) \in \mathbb{N}^{m k} \mid  w_i \in L_i \text{ for } 1 \leq i \leq m,  w_1 \cdots w_m =_G 1 \}
$$
is semilinear and a semilinear representation of this set can be computed.
Applying a projection yields a  semilinear representation of the set
\begin{align*}
\{(P(w_1), \ldots, P(w_n)) \in \mathbb{N}^{nk} \mid  \ & w_i \in u_i^* \text{ for } 1 \leq i \leq n,  \\
& \exists w  \in S : v_0 w_1 v_1 \cdots w_n v_n =_G w \} .
\end{align*}
Choose for every $u_i$ a symbol $a_{j_i} \in \Sigma$ such that
$\ell_i := |u_i|_{a_{j_i}} > 0$ (recall that $u_i \neq \varepsilon$). Then we project every $P(w_i)$ in the above set to the $j_i$-th coordinate.
The resulting projection is 
$$
\{(\ell_1 \cdot x_1, \ldots, \ell_n \cdot x_n) \in \mathbb{N}^n \mid \;  \exists w \in S :
v_0 u_1^{x_1} v_1 \cdots u_n^{x_n}  v_n =_G w \} .
$$
The semilinearity of this set easily implies the semilinearity of the set 
$$
\{(x_1, \ldots, x_n) \in \mathbb{N}^n \mid \;  \exists w \in S :
v_0 u_1^{x_1} v_1 \cdots u_n^{x_n} v_n =_G w \} .
$$
This concludes the proof.
\end{proof}
A subset $A \subseteq G$ is called {\em quasiconvex} if there exists a constant $\kappa \geq 0$ such that
every geodesic path from $1$ to some $g \in A$ is contained
in $\bigcup_{h \in A} B_\kappa(h)$.
The following result can be found in \cite{GeSh91} ($h$ denotes the evaluation morphism):

\begin{lem} \label{lemma-gersten-short}
A subset $A \subseteq G$ is quasiconvex if and only if the language of all geodesic words in $h^{-1}(A)$ is regular.
\end{lem}

\begin{thm} \label{theorem-quasiconvex}
Let $G$ be hyperbolic and let $H$ be a quasiconvex subgroup of $G$. Then $G$ is knapsack semilinear relative to $H$.
\end{thm}

\begin{proof}
Fix a finite symmetric generating set $\Sigma$ for $G$ and let $h$ be the evaluation morphism.
By Lemma~\ref{lemma-gersten-short} the set of all geodesic words in $h^{-1}(H)$ is a geodesic regular
language. By Theorem~\ref{theorem-knapsack-semilinear-rel.A}, $G$ is 
knapsack semilinear relative to $H$.
\end{proof}
It is known that every finitely generated free group $F$ is locally quasiconvex, which means that every
finitely generated subgroup of $F$ is quasiconvex.

\begin{cor}
Let $G$ be a finitely generated free group and let $H$ be a finitely generated subgroup of $G$. Then $G$ is knapsack semilinear relative to $H$.
\end{cor}
This corollary can actually be generalized. Schupp proved that a group $G$, which is virtually an orientable surface group of genus at least two
or virtually a Coxeter group satisfying a certain reduction hypothesis, is locally quasiconvex \cite[Theorem~IV]{Schupp03}. Since these groups are 
hyperbolic, it follows that the groups considered by Schupp are knapsack semilinear relative to any finitely generated subgroup.

Theorems~\ref{HNNsemilin} and \ref{theorem-quasiconvex} yield the following result:
\begin{cor}
Let $H=\langle G, t \mid t^{-1}at=a~(a\in A)\rangle$ be an HNN-extension where $G$ is hyperbolic and 
$A \leq G$ is a quasiconvex subgroup of $G$. 
Then $H$ is knapsack semilinear.
\end{cor}
It is known that every cyclic subgroup of a hyperbolic group is quasiconvex, see e.g.~\cite{Arz01}.
Hence, for every element $a \in G$ of a hyperbolic group $G$, the HNN-extension $\langle G, t \mid t^{-1}at=a\rangle$ 
is knapsack semilinear. It is also known that if the hyperbolic group $G$ is non-elementary (i.e., it contains a copy of the free group $F_2$) then
the centralizer of an element $g \in G$ is cyclic \cite[Lemma~2]{Arz01}. 
Hence, we obtain  Theorem~\ref{thm-ext-centralizer} for  the case of a centralizer of a single element in a
non-elementary hyperbolic group.

\section{Proof of Theorem~\ref{theorem-hyp-main0}}

We now come to the proof of Theorem~\ref{theorem-hyp-main0}. Let $G$ be $\delta$-hyperbolic.
For $1\leq i\leq n$ let $L_i \subseteq \Sigma^*$ be a regular $(\lambda,\epsilon)$-quasigeodesic  language. Let $\mathcal{A}_i = (\mathcal{Q}_i, S_i, \delta_i, T_i)$ be a finite
automaton for $L_i$. Without loss of generality, we can assume that every $q \in \mathcal{Q}_i$ belongs to a path from some initial state $q_0 \in S_i$ to some final state
$q_1 \in T_i$. This ensures that every word that labels a path from a state $p$
to a state $q$ is a factor of a word from $L_i$. Since factors of $(\lambda,\epsilon)$-quasigeodesic words are $(\lambda,\epsilon)$-quasigeodesic as well,
it follows that every word that labels a path between two states of $\mathcal{A}_i$ is $(\lambda,\epsilon)$-quasigeodesic.

We want to show that the set
\begin{equation*} 
\{ (P(w_1), \ldots, P(w_n)) \in \mathbb{N}^{nk} \mid w_i \in L_i \text{ for } 1 \leq i \leq n, w_1 w_2 \cdots w_n =_G 1  \}
\end{equation*}
is semilinear.  For this, we prove a slightly more general statement: For words $v_1, \ldots, v_n \in \Sigma^*$ we
consider the set
\begin{equation*} 
\{ (P(w_1), \ldots, P(w_n)) \in \mathbb{N}^{nk} \mid w_i \in L_i \text{ for } 1 \leq i \leq n, w_1 v_1 \cdots w_n v_n =_G 1  \}.
\end{equation*}
By induction over $n$ we show that this set is semilinear.  
For the case $n=2$ we can directly use Lemma~\ref{lemma-k=2-reg}. This also covers the case $n=1$ since we can take
$L_2 = \{1\}$.

Now assume that $n \geq 3$. We can assume that the words $v_i$ are geodesic. Define the automaton $\mathcal{A}$ as the disjoint union
of the  automata $\mathcal{A}_i$. Thus, the state set of $\mathcal{A}$ is $\mathcal{Q} = \biguplus_{1 \le i \le n} \mathcal{Q}_i$ and the transition set of 
$\mathcal{A}$ is $\delta = \biguplus_{1 \le i \le n} \delta_i$ (the sets of initial and final states of $\mathcal{A}$ are not important).
Let us denote for $p,q \in \mathcal{Q}$ with $L_{p,q}$ the set of all finite words that label a path from $p$ to $q$ in the
automaton $\mathcal{A}$. The above properties of the automata $\mathcal{A}_i$ ensure that every language 
 $L_{p,q}$ is $(\lambda,\epsilon)$-quasigeodesic.
Note that $L_i = \bigcup_{p \in S_i, q \in T_i} L_{p,q}$.
Since the semilinear sets are effectively closed  under union, it suffices to show for states $p_i,q_i \in \mathcal{Q}$ ($1 \leq i \leq n$) that the 
following set is semilinear:
\begin{equation*} 
\{ (P(w_1), \ldots, P(w_n)) \in \mathbb{N}^{nk} \mid w_i \in L_{p_i,q_i} \text{ for } 1 \leq i \leq n, w_1 v_1 \cdots w_n v_n =_G 1  \}.
\end{equation*}
In the following, we denote this set with $P(p_1,q_1,v_1, \ldots, p_n,q_n,v_n)$.
We will construct a Presburger formula with free variables $x_{i,j}$ ($1 \leq i \leq n$, $1 \leq j \leq k$)
for this set.  The variables $x_{i,j}$ with $1 \leq j \leq k$ encode the Parikh image of the words from $L_{p_i,q_i}$.
Let us write $\bar{x}_i = (x_{i,j})_{1 \leq j \leq k}$ in the following.

Recall from Section~\ref{sec-hyp} the definition of the path $P[w]$ for a word $w \in \Sigma^*$.
Consider a tuple $(w_1, \ldots, w_n) \in \prod_{i=1}^n L_{p_i,q_i}$ with
$w_1 v_1 w_2 v_2 \cdots w_n v_n =_G 1$
and the corresponding $2n$-gon that is defined by the $(\lambda,\epsilon)$-quasigeodesic paths
$P_i = (w_1 v_1 \cdots w_{i-1} v_{i-1}) \cdot P[w_i]$ and the geodesic paths
$Q_i = (w_1 v_1 \cdots w_{i}) \cdot P[v_i]$, see Figure~\ref{polygon} for the case $n=3$.
Since all paths $P_i$ and $Q_i$ are $(\lambda,\epsilon)$-quasigeodesic,
we can apply \cite[Lemma 6.4]{MyNiUs14}: Every side of the $2n$-gon is contained in the 
$\kappa$-neighborhoods of the other sides, where $\kappa = \xi+\xi\log(2n)$ for a constant $\xi$ that only depends
on the constants $\delta, \lambda, \varepsilon$.

\newlength{\R}\setlength{\R}{2.7cm}
\begin{figure}[t]
\centering
\begin{tikzpicture}
  [inner sep=.5mm,
  minicirc/.style={circle,draw=black,fill=black,thick}]

  \node (circ1) at ( 45:\R) [minicirc] {};
  \node (circ2) at (135:\R) [minicirc] {};
  \node (circ3) at (165:\R) [minicirc] {};
  \node (circ4) at (255:\R) [minicirc] {};
  \node (circ5) at (285:\R) [minicirc] {};
  \node (circ6) at (15:\R) [minicirc] {};
  \draw [thick] (circ1) to [out=225, in=-45] node[above=1mm] {$L_{p_2,q_2}$} 
                      (circ2) to [out=285, in=15] node[above=0.5mm,left=0.5mm] {$v_1$} 
                      (circ3) to [out=345, in=75] node[below=0.5mm,left=0.5mm] {$L_{p_1,q_1}$} 
                      (circ4) to [out=45, in=135] node[below=1mm] {$v_3$} 
                      (circ5) to [out=105, in=195] node[below=0.5mm,right=0.5mm] {$L_{p_3,q_3}$} 
                      (circ6) to [out=165, in=255] node[right=1mm] {$v_2$} (circ1);
\end{tikzpicture}
\caption{\label{polygon} The $2n$-gon for $n = 3$ from the proof of Theorem~\ref{theorem-hyp-main0}}
\end{figure}
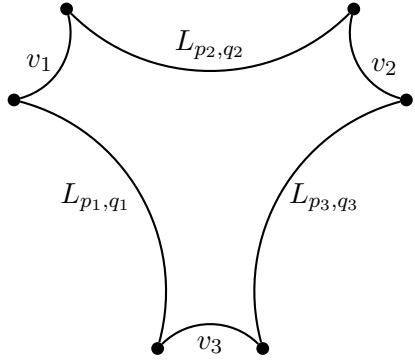

Let us now consider the side $P_2$ of the quasigeodesic $2n$-gon. It is labelled
with a word from $L_{p_2,q_2}$. Its neighboring  sides are $Q_1$ and $Q_2$, which are labelled
with $v_1$ and $v_2$, respectively.
We distinguish several cases. In each case we cut the $2n$-gon into smaller 
pieces along paths of length $\leq \kappa$ (length $2\kappa+1$ in Case~2). 
When we speak of a point on the $2n$-gon, we mean a node of the Cayley graph
(i.e., an element of the group $G$) and not a point in the interior of an edge.

For each of the following six cases we construct a Presburger formula describing a semilinear 
set. The union of these six sets is  $P(p_1,q_1,v_1, \ldots, p_n,q_n,v_n)$.

\medskip
\noindent
{\em Case 1:} There is a point $a \in P_2$ that has distance at most $\kappa$ from a point
$b$ that does not belong to $P_1 \cup Q_1 \cup Q_2 \cup P_3$. Thus $b$ must
belong to one of the paths $Q_3, P_4,  \ldots Q_{n-1},  P_n, Q_n$. Let $w$ be a geodesic
word of length at most $\kappa$ that labels a path from $a$ to $b$.
There are two subcases:

\medskip
\noindent
{\em Case 1.1:} $b$ belongs to a path $Q_i$ with $3 \le i \le n$.  
The situation is shown in Figure~\ref{polygon1.1} for $n=i=3$.
Let $T$ be the set of all tuples $(r, v_{i,1}, v_{i,2}, w)$ such that $r \in \mathcal{Q}$, $v_i = v_{i,1}v_{i,2}$, and $w \in \Sigma^*$ is of length at most $\kappa$.
By induction, the following two sets are semilinear for every tuple $t = (r, v_{i,1}, v_{i,2}, w) \in T$:
\begin{eqnarray*}
S_{t,1}  & = & P(p_1,q_1, v_1, p_2,r, wv_{i,2}, p_{i+1},q_{i+1}, v_{i+1}, \ldots, p_{n},q_{n}, v_{n}), \\
S_{t,2} & = & P(r, q_2, v_2, p_3, q_3, v_3, \ldots, p_i,q_i, v_{i,1} w^{-1}).
\end{eqnarray*}
Intuitively, $S_{t,1}$ corresponds to the $2j_1$-gon
(when $wv_{i,2}$ is viewed as a single side) on the left of the $w$-labelled edge
in Figure~\ref{polygon1.1}, whereas $S_{t,2}$ corresponds to the $2j_2$-gon on the right of the $w$-labelled edge (where $j_1= n-i+2$ and $j_2=i-1$). Note that $j_1,j_2 \leq n-1$.
We then define the formula
\begin{align*}
A_{1.1} =   \bigvee_{t \in T} \exists \bar{y}_{2},\bar{z}_{2} :  \ & (\bar{x}_1, \bar{y}_2, \bar{x}_{i+1}, \ldots, \bar{x}_n) \in S_{t,1}
\ \wedge \  (\bar{z}_2, \bar{x}_3, \ldots, \bar{x}_i) \in S_{t,2}  \ \wedge \  \bar{x}_2 = \bar{y}_{2} + \bar{z}_{2}.
\end{align*}
Here $\bar{y}_2$ and $\bar{z}_2$ are $k$-tuples of  new variables. The Presburger formula $A_{1.1}$ is one of the six formulas whose union
is $P(p_1,q_1,v_1, \ldots, p_n,q_n,v_n)$.

\medskip

\begin{figure}[t]
\centering
\begin{tikzpicture}
  [inner sep=.5mm,
  minicirc/.style={circle,draw=black,fill=black,thick}]
  \tikzstyle{small} = [circle,draw=black,fill=black,inner sep=.3mm]
  \node (circ1) at ( 45:\R) [minicirc] {};
  \node (circ2) at (135:\R) [minicirc] {};
  \node (circ3) at (165:\R) [minicirc] {};
  \node (circ4) at (255:\R) [minicirc] {};
  \node (circ5) at (285:\R) [minicirc] {};
  \node (circ6) at (15:\R) [minicirc] {};
  \draw [thick] (circ1) to [out=225, in=-45] 
     node[pos=.5,small]  (a) {} 
     node[pos=0.3,above=1.5mm] {$L_{r,q_2}$}
     node[pos=0.7,above=1.5mm] {$L_{p_2,r}$}
                      (circ2) to [out=285, in=15] node[above=0.5mm,left=0.5mm] {$v_1$} 
                      (circ3) to [out=345, in=75] node[below=0.5mm,left=0.5mm] {$L_{p_1,q_1}$} 
                      (circ4) to [out=45, in=135] node[below=1mm,pos=.3] {$v_{3,2}$} node[below=1mm,pos=.7] {$\; v_{3,1}$}  node[pos=.5,small]  (b){}
                      (circ5) to [out=105, in=195] node[below=0.5mm,right=0.5mm] {$L_{p_3,q_3}$} 
                      (circ6) to [out=165, in=255] node[right=1mm] {$v_2$} (circ1);
     \draw (a) to node[left=.8mm] {$w$} (b);
\end{tikzpicture}
\caption{\label{polygon1.1} Case~1.1 from the proof of Theorem~\ref{theorem-hyp-main0}}
\end{figure}
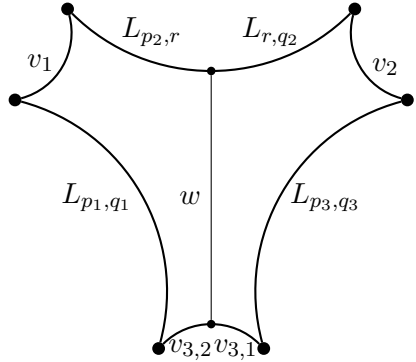

\noindent
{\em Case 1.2:} $b$ belongs to the path $P_i$, where $4 \le i \le n$
(this case can only occur if $n \ge 4$).
This case is analogous to Case 1.1. 
Let $T$ be the set of all tuples $(r,r', w)$ such that $r,r' \in \mathcal{Q}$ and $w \in \Sigma^*$ is of length at most $\kappa$.
By induction, the following two sets are semilinear for every tuple $t = (r,r', w) \in T$:
\begin{eqnarray*}
S_{t,1}  & = & P(p_1,q_1, v_1, p_2,r, w, r', q_i, v_i, p_{i+1},q_{i+1}, v_{i+1}, \ldots, p_{n},q_{n}, v_{n}), \\
S_{t,2}  & = & P(r, q_2, v_2, p_3, q_3, v_3, \ldots, p_{i-1}, q_{i-1}, v_{i-1}, p_i, r', w^{-1}).
\end{eqnarray*}
Moreover, let $A_{1.2}$ be the formula
\begin{align*}
A_{1.2} =   \bigvee_{t \in T} \exists \bar{y}_{2},\bar{z}_{2}, \bar{y}_{i},\bar{z}_{i} :  \ & (\bar{x}_1, \bar{y}_2, \bar{z}_{i}, \bar{x}_{i+1}, \ldots, \bar{x}_n) \in S_{t,1}
\ \wedge \\[-3mm]
&  (\bar{z}_2, \bar{x}_3, \ldots, \bar{x}_{i-1}, \bar{y}_i) \in S_{t,2}  \ \wedge \\
& \bar{x}_2 = \bar{y}_{2} + \bar{z}_{2} \ \wedge \ \bar{x}_i = \bar{y}_{i} + \bar{z}_{i} .
\end{align*}
{\em Case 2:} Every point on $P_2$ has distance at most $\kappa$ from a point
on $P_1 \cup Q_1 \cup Q_2 \cup P_3$. Since the starting point of $P_2$ has distance $0 \leq \kappa$ from $P_1 \cup Q_1$ and the end point of
$P_2$ has distance $0 \leq \kappa$ from $Q_2 \cup P_3$,  there must be points $b_1$ on $P_1 \cup Q_1$, $b$ on $P_2$, and $b_2$ on 
$Q_2 \cup P_3$ such that the distance between $b_1$ and $b$ is at most $\kappa$ and the distance between $b$ and $b_2$ is at most $\kappa+1$.
Hence, the distance between $b_1$ and $b_2$ is at most $2\kappa+1$.
Let $w$ be a word that labels a geodesic path from $b_1$ to $b_2$ (thus, $|w| \le 2\kappa+1$).
This leads to the following four subcases.

\medskip
\noindent
{\em Case 2.1:} $b_1 \in Q_1$ and $b_2 \in Q_2$. This case is shown in Figure~\ref{polygon2.3}.
Let $T$ be the set of all tuples $(v_{1,1},v_{1,2}, w, v_{2,1},v_{2,2})$ such that $v_1 = v_{1,1} v_{1,2}$,
$v_2 = v_{2,1} v_{2,2}$ and $w \in \Sigma^*$ is of length at most $2\kappa+1$.
By induction, the following two sets are semilinear for every tuple $t = (v_{1,1},v_{1,2}, w, v_{2,1},v_{2,2}) \in T$:
\begin{eqnarray*}
S_{t,1} & = & P(p_2,q_2, v_{2,1} w^{-1} v_{1,2}), \\
S_{t,2} & = & P(p_1, q_1, v_{1,1} w v_{2,2}, p_3,q_3,v_3, \ldots,  p_n,q_n,v_n).
\end{eqnarray*}
We define the  formula
\[
A_{2.1} =  \bigvee_{t \in T}  \bar{x}_2 \in S_{t,1}
\ \wedge \  (\bar{x}_1, \bar{x}_3 \ldots, \bar{x}_n) \in S_{t,2}  .
\]

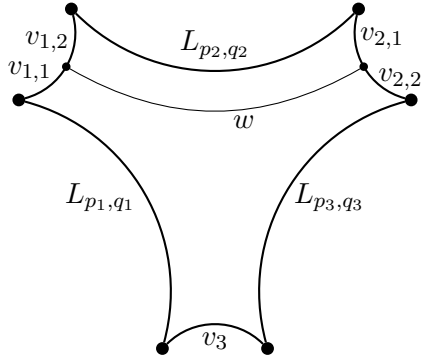
\begin{figure}[t]
\centering
\begin{tikzpicture}
  [inner sep=.5mm,
  minicirc/.style={circle,draw=black,fill=black,thick}]
  \tikzstyle{small} = [circle,draw=black,fill=black,inner sep=.3mm]
  \node (circ1) at ( 45:\R) [minicirc] {};
  \node (circ2) at (135:\R) [minicirc] {};
  \node (circ3) at (165:\R) [minicirc] {};
  \node (circ4) at (255:\R) [minicirc] {};
  \node (circ5) at (285:\R) [minicirc] {};
  \node (circ6) at (15:\R) [minicirc] {};
  \draw [thick] (circ1) to [out=225, in=-45] node[above=1mm] {$L_{p_2,q_2}$} 
                      (circ2) to [out=285, in=15] node[left=.1mm,pos=.25] {$v_{1,2}$} node[above=.5mm,left=.5mm,pos=.6] {$v_{1,1}$}  node[pos=.5,small]  (b1){}
                      (circ3) to [out=345, in=75] node[below=0.5mm,left=0.5mm] {$L_{p_1,q_1}$} 
                      (circ4) to [out=45, in=135] node[below=0.5mm] {$v_3$}
                      (circ5) to [out=105, in=195] node[below=0.5mm,right=0.5mm] {$L_{p_3,q_3}$} 
                      (circ6) to [out=165, in=255] node[right=.1mm,pos=.35] {$v_{2,2}$} node[above=.5mm,right=0mm,pos=.75] {$v_{2,1}$}  node[pos=.5,small]  (b2){}
                      (circ1);
     \draw (b1) to [out=-30, in=210] node[below=.3mm,pos=.6] {$w$} (b2);
\end{tikzpicture}
\caption{\label{polygon2.3} Case~2.1 from the proof of Theorem~\ref{theorem-hyp-main0}}
\end{figure}

\medskip
\noindent
{\em Case 2.2:} $b_1 \in P_1$ and $b_2 \in Q_2$, see Figure~\ref{polygon2.4}.
This case is exactly the same as Case~1.1 with $i=3$, if we replace the side $P_2$ in Case~1.1 by $P_1$; see
Figure~\ref{polygon1.1}.

\begin{figure}[t]
\centering
\begin{tikzpicture}
  [inner sep=.5mm,
  minicirc/.style={circle,draw=black,fill=black,thick}]
  \tikzstyle{small} = [circle,draw=black,fill=black,inner sep=.3mm]
  \node (circ1) at ( 45:\R) [minicirc] {};
  \node (circ2) at (135:\R) [minicirc] {};
  \node (circ3) at (165:\R) [minicirc] {};
  \node (circ4) at (255:\R) [minicirc] {};
  \node (circ5) at (285:\R) [minicirc] {};
  \node (circ6) at (15:\R) [minicirc] {};
  \draw [thick] (circ1) to [out=225, in=-45] node[above=1mm] {$L_{p_2,q_2}$} 
                      (circ2) to [out=285, in=15] node[above=0.5mm,left=0.5mm] {$v_1$} 
                      (circ3) to [out=345, in=75] 
                      node[pos=.5,small]  (b1) {} 
     node[pos=0.33,left=1mm] {$L_{r,q_1}$}
     node[pos=0.7,left=-.5mm] {$L_{p_1,r}$}
                      (circ4) to [out=45, in=135] node[below=0.5mm] {$v_3$}
                      (circ5) to [out=105, in=195] node[below=0.5mm,right=0.5mm] {$L_{p_3,q_3}$} 
                      (circ6) to [out=165, in=255] node[right=.1mm,pos=.35] {$v_{2,2}$} node[above=.5mm,right=0mm,pos=.75] {$v_{2,1}$}  node[pos=.5,small]  (b2){}
                      (circ1);
     \draw (b1) to [out=30, in=210] node[below=.3mm,pos=.6] {$w$} (b2);
\end{tikzpicture}
\caption{\label{polygon2.4} Case~2.2 from the proof of Theorem~\ref{theorem-hyp-main0}}
\end{figure}
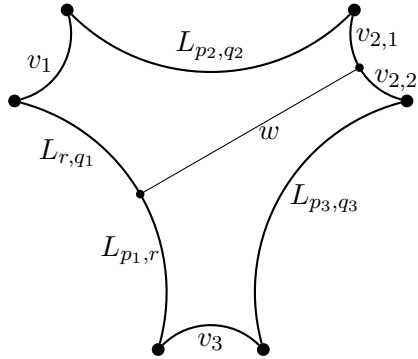

\medskip
\noindent
{\em Case 2.3:} $b_1 \in Q_1$ and $b_2 \in P_3$.
This case is analogous to Case~2.2.

\medskip
\noindent
{\em Case 2.4:} $b_1 \in P_1$ and $b_2 \in P_3$, see Figure~\ref{polygon2.6}.
Let $T$ be the set of all tuples $(w_1, w_2, w, r_1, r_2, r_3)$
such that $|w| \leq 2\kappa+1$, $|w_1| \leq \kappa$, $|w_2| \leq \kappa+1$, $w = w_1^{-1} w_2$ in $G$,
and $r_1, r_2, r_3 \in \mathcal{Q}$.
By induction, the following three sets are semilinear for every tuple $t = (w_1, w_2, w, r_1, r_2, r_3) \in T$:
\begin{eqnarray*}
S_{t,1} & = & P(r_1,q_1,v_1,p_2,r_2,w_1), \\
S_{t,2} & = & P(r_2,q_2,v_2,p_3,r_3,w_2^{-1}), \\
S_{t,3} & = & P(p_1,r_1,w,r_3,q_3, v_3,p_4, q_4,v_4,\ldots, p_n, q_n,v_n).
\end{eqnarray*}
 We define the formula
\begin{align*}
A_{2.4} = \bigwedge_{t \in T} \exists \bar{y}_1, \bar{z}_1, \bar{y}_2, \bar{z}_2, \bar{y}_3, \bar{z}_3 : \  & 
(\bar{z}_1,\bar{y}_2)  \in S_{t,1} \ \wedge \ (\bar{z}_2,\bar{y}_3)  \in S_{t,2} \ \wedge  \\[-3mm]
& (\bar{y}_1,\bar{z}_3, \bar{x}_4, \ldots, \bar{x}_n) \in S_{t,3} \ \wedge \ \bigwedge_{i=1}^3 \bar{x}_i = \bar{y}_i + \bar{z}_i .
\end{align*}

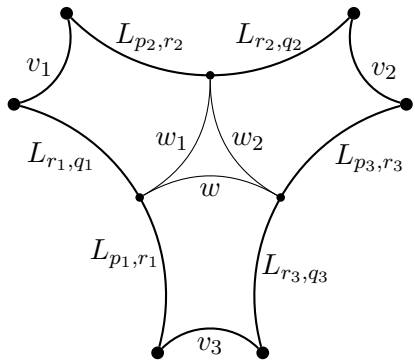
\begin{figure}[t]
\centering
\begin{tikzpicture}
  [inner sep=.5mm,
  minicirc/.style={circle,draw=black,fill=black,thick}]
  \tikzstyle{small} = [circle,draw=black,fill=black,inner sep=.3mm]
  \node (circ1) at ( 45:\R) [minicirc] {};
  \node (circ2) at (135:\R) [minicirc] {};
  \node (circ3) at (165:\R) [minicirc] {};
  \node (circ4) at (255:\R) [minicirc] {};
  \node (circ5) at (285:\R) [minicirc] {};
  \node (circ6) at (15:\R) [minicirc] {};
  \draw [thick] (circ1) to [out=225, in=-45] 
     node[pos=.5,small]  (a) {} 
     node[pos=0.3,above=1.3mm] {$L_{r_2,q_2}$}
     node[pos=0.7,above=1.3mm] {$L_{p_2,r_2}$}
                      (circ2) to [out=285, in=15] node[above=0.5mm,left=0.5mm] {$v_1$} 
                      (circ3) to [out=345, in=75] 
                      node[pos=.5,small]  (b1) {} 
     node[pos=0.33,left=1mm] {$L_{r_1,q_1}$}
     node[pos=0.7,left=-.5mm] {$L_{p_1,r_1}$}
                      (circ4) to [out=45, in=135] node[below=0.5mm] {$v_3$}
                      (circ5) to [out=105, in=195] 
                      node[pos=.5,small]  (b2) {} 
     node[pos=0.25,right=0mm] {$L_{r_3,q_3}$}
     node[pos=0.7,below=1mm,right=1.5mm] {$L_{p_3,r_3}$}
                      (circ6) to [out=165, in=255] node[right=1mm] {$v_2$}
                      (circ1);
        \draw (a) to [out=-90, in=30] node[left=.5mm,pos=.45] {$w_1$} (b1);
        \draw (a) to [out=-90, in=150] node[right=.5mm,pos=.45] {$w_2$} (b2);
        \draw (b1) to [out=30, in=150] node[below=.5mm,pos=.5] {$w$} (b2);
\end{tikzpicture}
\caption{\label{polygon2.6} Case~2.4 from the proof of Theorem~\ref{theorem-hyp-main0}}
\end{figure}

\medskip
\noindent
This concludes the case distinction.

A tuple $(\bar{x}_1, \ldots, \bar{x}_n) \in \N^{n k}$ belongs to the set $P(p_1,q_1,v_1, \ldots, p_n,q_n,v_n)$ if and only if
$A_{1.1} \vee A_{1.2} \vee A_{2.1} \vee A_{2.2} \vee A_{2.3} \vee A_{2.4}$ holds.
This yields a Presburger formula for $P(p_1,q_1,v_1, \ldots, p_n,q_n,v_n)$.
 \qed

\section*{Acknowledgment}
  \noindent 
Both authors were supported by the DFG research project LO 748/12-1.

\bibliographystyle{plainurl}

\def\cprime{$'$} \def\cprime{$'$}

\end{document}